\documentclass[12pt]{amsart}
\usepackage{amssymb}

\newtheorem{theorem}{Theorem}[section]
\newtheorem{lemma}[theorem]{Lemma}

\newtheorem{corollary}[theorem]{Corollary}
\newtheorem{proposition}[theorem]{Proposition}

\newtheorem{lem-def}[theorem]{Lemma-Definition}

\newcommand{\hooklongrightarrow}{\lhook\joinrel\longrightarrow}
\renewenvironment{proof}{{\bfseries Proof.}}{\qed}
\newcommand{\R}{\mathbb R}

\newcommand{\M}{\mathbb M}
\newcommand{\N}{\mathbb N}

\newcommand{\Q}{\mathbb Q}

\newcommand{\T}{\mathbb T}

\def\op{\operatorname}

\def\aa{\mathcal{A}}

\def\al{\alpha}

\def\ars#1{\renewcommand\arraystretch{#1}}

\def\bb{{\mathcal B}}

\def\be{\beta}

\def\cfa{\left(\ri\right)_{i\in A}}

\def\chr{\op{char}}

\def\cl#1{\left[\,#1\,\right]_\mu}
\def\clok#1{\left[\,#1\,\right]_{\op{ok}}}

\def\defn{\nn{\bf Definition. }}

\def\dg{\op{deg}}
\def\dgi{\op{deg}_\infty}

\def\diso{\lower.4ex\hbox{$\downarrow$}\raise.4ex\hbox{\mbox{\scriptsize
$\wr$}}}

\def\dta{\delta}

\def\e{\medskip}

\def\ep#1{\exp(\Pi i#1)}
\def\ep{\epsilon}

\def\f{\mathfrak{f}}

\def\fin{{\op{fin}}}

\def\g{\Gamma}
\def\ga{\gamma}
\def\gaal{\ga_{\aa_\ell}}
\def\gal{\op{Gal}}

\def\gen#1{\big\langle\, {#1} \,\big\rangle}
\def\gf{\mathfrak{g}}

\def\ggm{\mathcal{G}_\mu}

\def\gm{\g_\mu}

\def\gq{\g_\Q}
\def\gr{\operatorname{gr}}

\def\gsme{\g_{\op{sme}}}

\def\h{\mathfrak{h}}

\def\hk{\hookrightarrow}

\def\hra{\hooklongrightarrow}

\def\imp{\ \Longrightarrow\ }

\def\ind{\op{ind}}

\def\inm{\op{in}_\mu}

\def\irr{\op{Irr}}
\def\ism{\lower.3ex\hbox{\ars{.08}$\begin{array}{c}\,\to\\\mbox{\tiny $\sim\,$}\end{array}$}}
\def\iso{\ \lower.3ex\hbox{\ars{.08}$\begin{array}{c}\lra\\\mbox{\tiny $\sim\,$}\end{array}$}\ }

\def\ka{\kappa}
\def\kam{\kappa(\mu)}
\def\kb{\overline{K}}
\def\kh{K^h}
\def\khx{K^h[x]}
\def\km{k_\mu}

\def\kp{\op{KP}}

\def\kpi{\op{KP}_\infty}
\def\kpm{\op{KP}(\mu)}
\def\kpmz{\kp_{\op{str}}(\mu)}

\def\kpr{\op{KP}(\rho)}

\def\ks{K^{\op{sep}}}

\def\kx{K[x]}

\def\La{\Lambda}

\def\lfin{\ll_\fin}
\def\lg{l\raise.6ex\hbox to.2em{\hss.\hss}l}

\def\ll{\mathcal{L}}
\def\lra{\,\longrightarrow\,}

\def\mal{\mu_{\aa_\ell}}

\def\minf{\om_{-\infty}}

\def\mlv{Mac Lane--Vaqui\'e\ }

\def\mmu{\mid_\mu}

\def\nn{\noindent}

\def\ok{\sim_{\op{ok}}}
\def\om{\omega}

\def\orb{\hbox to  .3em{$\backslash$}\backslash}

\def\p{\mathfrak{p}}

\def\prh{\pset(\rho)}

\def\pmph{\pset_\mu(\phi)}

\def\ppa{\mathcal{P}_{\alpha}}
\def\prim{\op{Prim}}

\def\pset{\mathcal{P}}

\def\ral{\rho_{\aa_\ell}}

\def\rf{\rho(\f)}
\def\rf{\rho(\f)}
\def\rfp{\rho(\f')}

\def\rg{\rho(\mathfrak{g})}
\def\rha{\rho_\aa}

\def\ri{\rho_i}

\def\rj{\rho_j}

\def\rsme{\R_\sme}

\def\rrk{\op{rr}}

\def\sii{\ \Longleftrightarrow\ }

\def\sme{{\mbox{\tiny $\op{sme}$}}}
\def\smu{\sim_\mu}

\def\sind{^{\ind}}

\def\sp{\op{Spec}}

\def\supp{\op{supp}}

\def\sval{\operatorname{sv}}
\def\t{\theta}

\def\td{\mbox{\rm\bf td}}
\def\tdm{\td(\mu)}

\def\tla{\ttt(\La)}
\def\tmn{\ty(\mu,\nu)}

\def\tprim{\T^{\op{prim}}}

\def\tq{\ttt_\Q}

\def\ts{\mathcal{T}_\sme}
\def\ttt{\mathcal{T}}
\def\ty{\mathbf{t}}

\def\vb{\overline{v}}

\def\wb{\overline{w}}
\def\wt{\op{wt}}

\newcounter{cs}
\stepcounter{cs}
\newcommand{\casos}{\begin{itemize}}
\newcommand{\fcasos}{\end{itemize}\setcounter{cs}{1}}

\newfont{\tit}{cmr12 scaled \magstep3}

\title[Okutsu frames]{Okutsu frames of irreducible polynomials over henselian fields}

\makeatletter
\@namedef{subjclassname@2010}{%
  \textup{2010} Mathematics Subject Classification}

\author[Alberich]{Maria Alberich-Carrami$\tilde{\mbox{n}}$ana}
\address{Institut de Rob\`otica i Inform\`atica Industrial (IRI, CSIC-UPC), Institut de Mate\-mà\-tiques de la UPC-BarcelonaTech (IMTech) and Departament de Mate\-m\`a\-tiques, Universitat Polit\`ecnica de Cata\-lunya $\cdot$ BarcelonaTech, Av. Diagonal, 647, E-08028 Barcelona, Catalonia}
\email{Maria.Alberich@upc.edu}
\author[Gu\`ardia]{Jordi Gu\`ardia}
\address{Departament de Matem\`atiques, Escola Polit\`ecnica Superior d'Enginye\-ria de Vilanova i la Geltr\'u, Av. V\'\i ctor Balaguer s/n. E-08800 Vilanova i la Geltr\'u, Catalonia}
\email{jordi.guardia-rubies@upc.edu}
\author[Nart]{Enric Nart}
\author[Ro\'e]{Joaquim Ro\'e}
\address{Departament de Matem\`{a}tiques,         Universitat Aut\`{o}noma de Barcelona,         Edifici C, E-08193 Bellaterra, Barcelona, Catalonia}
\email{nart@mat.uab.cat,\quad jroe@mat.uab.cat}

\thanks{Partially supported by grants  PID2020-116542GB-I00 and PID2019-103849GB-I00 from the Spanish Research Agency, and grant 2017SGR-932 from Generalitat de Catalunya}
\date{\today}

\begin{document}


\begin{abstract}
For a henselian valued field $(K,v)$ we establish a complete parallelism between the arithmetic properties of irreducible polynomials $F\in \kx$,  encoded by their Okutsu frames, and the valuation-theoretic properties of their induced valuations $v_F$ on $\kx$, encoded by their Mac Lane-Vaqui\'e chains. This parallelism was only known for defectless irreducible polynomials.
\end{abstract}

\maketitle



\section*{Introduction}
The pioneering work of Mac Lane on valuations on a polynomial ring \cite{mcla},  was inspired in a question of Ore about the design of an algorithm to compute prime ideal decomposition in number fields \cite{ore1}. To solve this question, Mac Lane used the methods of \cite{mcla} to develop a polynomial factorization algorithm over the completion  $K_v$ of any discrete rank-one valued field $(K,v)$ \cite{mclb}. The algorithm finds  \emph{key polynomials} of certain valuations on $\kx$, as approximations to the irreducible factors in $K_v[x]$ of any given separable polynomial in $\kx$. 

Motivated by the computation of integral bases in finite extensions of local fields, Okutsu constructed  similar approximations without using valuations on $\kx$, nor key polynomials \cite{Ok}. Each irreducible polynomial $F\in K_v[x]$ admits an \emph{Okutsu frame}, which is a finite list  of polynomials which are best possible approximations to $F$ among all polynomials with degree smaller than a certain bound.

Fern\'andez, Gu\`ardia, Montes and Nart showed that the approaches of Ore-Mac Lane and Okutsu are essentially equivalent in the discrete rank-one case \cite{ResidualIdeals,okutsu}.

The techniques of Mac Lane in \cite{mcla}
were extended to arbitrary valued fields independently by Vaqui\'e \cite{Vaq} and Herrera-Mahboub-Olalla-Spivakovsky \cite{hos,hmos}. 
The extension of \cite{mclb} to a polynomial factorization algorithm over arbitrary henselian fields is still an open problem. 

Let $(K,v)$ be a henselian valued field of an arbitrary rank. Let
$\irr(K)$ be the set of all monic irreducible polynomials in $\kx$.
 Every $F\in\irr(K)$ determines a valuation $v_F$ on $\kx$ given by $$v_F(f)=v(f(\t)),\quad\mbox{for all }f\in\kx,$$where $\t$ is any root of $F$ in an algebraic closure of $K$. 

 This valuation $v_F$ is the end node of some finite \emph{Mac Lane-Vaqui\'e (MLV) chain} of valuations on $K(x)$ 
 $$
 \mu_0\,\lra\,\mu_1\,\lra\,\cdots\,\lra\,\mu_r\,\lra\,v_F
 $$
 where $\mu_0$ admits key polynomials of degree one, and every step $\mu_i\to\mu_{i+1}$ may be either an ordinary or a limit augmentation. 
 
 We say that $F\in\irr(K)$ is \emph{defectless} if the unique extension $w$ of $v$ to $\kx/(F)$ satisfies $\deg(F)=e(w/v)f(w/v)$; or, equivalently, if the MLV chain of $v_F$ contains only ordinary augmentations \cite{Vaq2,defless}. 

The results of \cite{ResidualIdeals,okutsu} were generalized in \cite{defless} to approximate  defectless polynomials over arbitrary henselian fields.
A certain \emph{Okutsu equivalence relation} $\ok$ was defined on the set $\op{Dless}$  of all defectless polynomials, so that the quotient set  $\op{Dless}/\!\!\ok$ admits a parametrization by a \emph{Mac Lane space} $\M$ described in terms of valuations on $\kx$ \cite[Thm. 5.14]{defless}. 
On the other hand, a complete parallelism was established between the arithmetic properties of any $F\in\op{Dless}$,  encoded by their Okutsu frames, and the valuation-theoretic properties of the valuation $v_F$, encoded by their MLV chains \cite[Thms. 5.5, 5.6]{defless}.

In this paper, we apply the methods of \cite[Secs. 6,7]{VT} on valuative trees to generalize all these results
to arbitrary irreducible polynomials in $\kx$. The main obstacle is the presence of limit augmentations in the MLV chains of $v_F$. This implies that there are no longer ``best possible" approximations to $F$ of a given bounded degree, so that Okutsu frames need to be reformulated.

The structure of the paper is as follows. Section \ref{secVals} collects preliminary results on key polynomials, valuative trees and their tangent spaces. Let $\tq$ be the tree whose nodes are extensions of $v$ to valuations on $\kx$ taking values in the divisible hull of the group $v(K^*)$. This tree admits a certain \emph{small-extensions closure} $\tq\subset\ts$ whose tangent space  plays a key role.  

In Section \ref{secOk}, we generalize \cite[Thm. 5.14]{defless}. We introduce an Okutsu equivalence relation $\ok$ on the set $\lfin(\tq)$ of finite leaves of the tree $\tq$. After a natural identification of $\lfin(\tq)$ with $\irr(K)$, the quotient set $\irr(K)/\!\!\ok$ may be parametrized by the set $\tprim$ of primitive tangent vectors in the tangent space of $\ts$. 

In Section \ref{secOkHensel} we show that two irreducible polynomials are Okutsu equivalent if and only if they are ``sufficiently close" with respect to the classical ultrametric topology induced by $v$. Also, we see that the subset of $\tprim$ corresponding to defectless polynomials may be identified with the Mac Lane space $\M$.

In Section \ref{secOkFrames}, an \emph{Okutsu frame} of any $F\in\irr(K)$ is defined as a list of sets of polynomials: $\left[\Phi_0,\dots,\Phi_r,\Phi_{r+1}=\{F\}\right]$.
Each $\Phi_i$ contains monic polynomials $f\in \kx$ of constant degree $m_i$, whose weighted values $v_F(f)/\deg(f)$ are cofinal in the set of all weighted $v_F$-values of polynomials of degree less than $m_{i+1}$. These degrees satisfy
$$
1=m_0\,\mid\,m_1\,\mid\,\cdots\,\mid\,m_r\,\mid\,m_{r+1}=\deg(F).
$$

Theorems \ref{MLOk} and \ref{OkML} show that the Okutsu frames of $F$ and the MLV chains of $v_F$ are essentially equivalent objects.
As a consequence, we obtain still another interpretation of Okutsu frames: 
the set $\Phi_0\cup\cdots\cup\Phi_r\cup\{F\}$ is a complete set of \emph{abstract} key polynomials of $v_F$, in the terminology of \cite{Dec, NS2018}.



\section{Valuative trees}\label{secVals}

In this section we recall some  results on valuations on a polynomial ring, valuative trees and their tangent spaces.

Let $(K,v)$ be a valued field. Let $k$ be the residue class field, $\g=v(K^*)$ the value group and $\gq=\g\otimes\Q$ the divisible hull of $\g$. The \emph{rational rank} of $\g$ is  defined as $\rrk(\g)=\dim_\Q \gq$.

Let $\g\hk\Lambda$ be an extension of ordered abelian groups. We write simply $\Lambda\infty$ instead of $\Lambda\cup\{\infty\}$.
Consider a valuation on the polynomial ring $\kx$
$$
\mu\colon \kx\lra \Lambda\infty
$$
whose restriction to $K$ is $v$. 

Let $\p=\supp(\mu):=\mu^{-1}(\infty)\in\sp(\kx)$ be the support of $\mu$. 
The valuation $\mu$ induces in a natural way a valuation $\overline{\mu}$ on the field of fractions of $\kx/\p$, which is $K(x)$ if $\p=0$, or $\kx/\p$ if $\p=f\kx$ for some $f\in\irr(K)$.

The residue field $\km$ of $\mu$ is, by definition, the residue field of $\overline{\mu}$. The field $\kam$ of \emph{algebraic residues} of $\mu$ is the relative algebraic closure of $k$ in $\km$.

We say that $\mu$ is \emph{commensurable} (over $v$) if $\g_\mu/\g$ is a torsion group; or equivalently, $\rrk(\gm/\g)=0$. In this case, there is a canonical embedding $\g_\mu\hookrightarrow \gq$. 

All valuations on $\kx$ satisfy \emph{Abhyankar's inequality}
$$
\rrk(\gm/\g)+\op{trdeg} (\km/k)\le 1.
$$
This yields a basic classification of these valuations in three families:

\begin{itemize}
	\item $\mu$ is \emph{valuation-algebraic} if \ $\rrk(\gm/\g)=\op{trdeg} (\km/k)=0$.
	\item $\mu$ is incommensurable if \ $\rrk(\gm/\g)=1$.
	\item $\mu$ is \emph{residue transcendental} if \ $\op{trdeg} (\km/k)=1$. 
\end{itemize}

The valuations for which equality holds in Abhyankar's inequality, corresponding to the two latter families, are said to be \emph{valuation-transcendental}.
   
All valuations with nontrivial support are valuation-algebraic, but the converse statement is false. 

\subsection{Key polynomials}\label{subsecKP}

For any $\alpha\in\g_\mu$, consider the abelian groups:
$$
\ppa=\{g\in \kx\mid \mu(g)\ge \alpha\}\supset
\ppa^+=\{g\in \kx\mid \mu(g)> \alpha\}.
$$    
The \emph{graded algebra of $\mu$} is the integral domain:
$$
\ggm:=\gr_{\mu}(\kx)=\bigoplus\nolimits_{\alpha\in\g_\mu}\ppa/\ppa^+.
$$

There is an \emph{initial coefficient} mapping $\inm\colon \kx\to \ggm$, given by $\inm\p=0$ and 
$$
\inm g= g+\pset_{\mu(g)}^+\in\pset_{\mu(g)}/\pset_{\mu(g)}^+, \qquad\mbox{if }\ g\in \kx\setminus\p.
$$

 The following definitions translate properties of the action of  $\mu$ on $\kx$ into algebraic relationships in the graded algebra $\ggm$.\e

\defn Let $g,\,h\in \kx$.

We say that $g,h$ are \emph{$\mu$-equivalent}, and we write $g\smu h$, if $\inm g=\inm h$. 

We say that $g$ is \emph{$\mu$-divisible} by $h$, and we write $h\mmu g$, if $\inm h\mid \inm g$ in $\ggm$.

We say that $g$ is $\mu$-irreducible if $\inm g$ is a prime element; that is, it generates a nonzero homogeneus prime ideal. 

We say that $g$ is $\mu$-minimal if $g\nmid_\mu f$ for all nonzero $f\in \kx$ with $\deg(f)<\deg(g)$.\e


For all $g\in\kx\setminus K$ we define the \emph{truncation} $\mu_g$ as follows: 
$$f=\sum\nolimits_{0\le s}a_s g^s,\quad \deg(a_s)<\deg(g)\ \imp\ \mu_g(f)=\min\left\{\mu\left(a_sg^s\right)\mid 0\le s\right\}.$$
This function $\mu_g$ is not necessarily a valuation, but it is useful to characterize the  $\mu$-minimality of $g$. Let us recall \cite[Prop. 2.3]{KP}.

\begin{lemma}\label{minimal0}
A polynomial  $g\in \kx\setminus K$ is $\mu$-minimal if and only if $\mu_g=\mu$.
\end{lemma}

\defn A  \emph{(Mac Lane-Vaqui\'e) key polynomial} for $\mu$ is a monic polynomial in $\kx$ which is simultaneously  $\mu$-minimal and $\mu$-irreducible. 

The set of key polynomials for $\mu$ is denoted $\kpm$. \e

All $\phi\in\kpm$ are irreducible in $\kx$. 
A \emph{tangent direction} of $\mu$ is a $\mu$-equivalence class $\cl{\phi}\subset\kpm$ determined by all key polynomials having the same initial coefficient in $\ggm$. We denote the set of all tangent directions of $\mu$ by:
$$
\tdm=\kpm/\!\smu.
$$
Since all polynomials in $\cl{\phi}$ have the same degree \cite[Prop. 6.6]{KP}, it makes sense to consider the degree $\deg\, [\phi]_\mu$ of a tangent direction.

The basic families of valuations can be characterized as follows by their tangent directions \cite[Thms. 1.2, 1.4]{VT}.

\begin{itemize}
	\item If $\mu$ is valuation-algebraic, then $\tdm=\emptyset$. That is, $\kpm=\emptyset$.
	\item If $\mu$ is incommensurable, then $\tdm$ is a one-element set.
	\item If $\mu$ is residue transcendental, then $\tdm$ is in bijection with $\irr(\ka(\mu))$. 
\end{itemize}

In the latter case, the bijection is determined by the choice of a key polynomial of minimal degree and a certain homogeneous unit in $\ggm$ \cite[Sec. 6]{KP}.\e  

\defn
Suppose that  $\kpm\ne\emptyset$ and take $\phi\in\kpm$ of minimal degree. The following data are independent of the choice of $\phi$:
$$
\deg(\mu)=\deg(\phi),\qquad \sval(\mu)=\mu(\phi),\qquad \wt(\mu)=\sval(\mu)/\deg(\mu).
$$
They are called the \emph{degree}, the \emph{singular value} and the \emph{weight} of $\mu$, respectively. 

\begin{theorem}\cite[Thm. 3.9]{KP}\label{univbound}
If $\kpm\ne\emptyset$, then for all monic $f\in \kx\setminus K$, we have $\,\mu(f)/\deg(f)\le \wt(\mu)$. Equality holds if and only if $f$ is $\mu$-minimal.
\end{theorem}

\subsection{Tangent space of a valuative tree}\label{subsecTangent}
Let $\ttt=\tla$ be the set of all valuations $$\mu\colon \kx\lra\La\infty,$$ whose restriction to $K$ is $v$.
This set admits a partial ordering. 
For $\mu,\nu\in \ttt$ we define  
$$\mu\le\nu\ \sii\ \mu(f)\le \nu(f), \quad\forall\,f\in\kx.$$
As usual, we write $\mu<\nu$ to indicate that $\mu\le\nu$ and $\mu\ne\nu$.


This poset $\ttt$ has the structure of a tree. By this, we simply mean that the intervals 
$$
(-\infty,\mu\,]:=\left\{\rho\in\ttt\mid \rho\le\mu\right\}
$$
are totally ordered for all $\mu\in\ttt$ \cite[Thm. 2.4]{MLV}. 

A node $\mu\in\ttt$ is a \emph{leaf} if it  is a maximal element with respect to the ordering $\le$. Otherwise, we say that $\mu$ is an \emph{inner node}.

\begin{theorem}\cite[Thm. 2.3]{MLV}\label{maximal}
	A node $\mu\in\ttt$ is a leaf if and only it $\kpm=\emptyset$.
\end{theorem}

All valuations with nontrivial support are leaves of $\ttt$. We call them \emph{finite leaves}.
The leaves of $\ttt$ having trivial support are called \emph{infinite leaves}.

A finite leaf  $\nu\in\ttt$  has $\supp(\nu)=F\kx$ for some   $F\in\irr(K)$. We define $$\deg(\nu)=\deg(F),\qquad\sval(\nu)=\wt(\nu)=\infty.
$$ 

The following result recalls some fundamental properties of tangent directions. It follows from \cite[Thm. 1.15]{Vaq},  \cite[Prop. 2.2]{MLV} and \cite[Prop. 2.5]{VT}.

\begin{lemma}\label{propertiesTMN}
	Let $\mu<\nu$ be two nodes in $\ttt$. Let $\tmn$ be the set of all monic polynomials $\phi\in\kx$ of minimal degree satisfying $\mu(\phi)<\nu(\phi)$.

\ (i) \ \ If $\nu$ is an inner node or a finite leaf, then $\,\deg(\mu)\le \deg(\nu)$ and $\,\wt(\mu)< \wt(\nu)$.

\,(ii) \,The set $\tmn$ is a tangent direction of $\mu$. 
Moreover, for all $\phi\in\tmn$, $f\in\kx$, the equality $\mu(f)=\nu(f)$ holds if and only if $\phi\nmid_{\mu}f$.   

(iii) \,If $\mu<\nu'$ for some $\nu'\in\ttt$, then
 
\qquad\qquad\qquad$\tmn=\ty(\mu,\nu')\sii (\mu,\nu\,]\cap (\mu,\nu'\,]\ne\emptyset$. 
\end{lemma}

\defn The \emph{tangent space} of $\ttt$ is the set
$$
\T(\ttt)=\left\{(\mu,t)\mid \mu\mbox{ inner node in }\ttt,\ t\in\tdm\right\}.
$$

\subsection{\mlv chains}\label{subsecMLV}


For any $\phi\in\kpm$ and any $\ga\in\Lambda\infty$ such that $\mu(\phi)<\ga$, we can construct the \emph{ordinary augmented valuation}  $\nu=[\mu;\,\phi,\ga]\in\ttt$,
defined in terms of $\phi$-expansions as
$$
 f=\sum\nolimits_{0\le s}a_s\phi^s,\quad \deg(a_s)<\deg(\phi)\ \imp\ \nu(f)=\min\nolimits_{0\le s}\{\mu(a_s)+s\ga\}.
$$
Note that $\nu(\phi)=\ga$, $\mu<\nu$ and $\ty(\mu,\nu)=\cl{\phi}$. 

If $\ga<\infty$, then $\phi$ is a key polynomial for $\nu$ of minimal degree \cite[Cor. 7.3]{KP}. 

Let $\aa=\cfa\subset\ttt$ be a totally ordered family 
not admitting a maximal element. Assume that $\aa$ is
parametrized by a totally ordered set $A$ of indices such that  the mapping
$A\to\aa$ determined by $i\mapsto \ri$
is an isomorphism of ordered sets.

If $\deg(\rho_{i})$ is stable for all sufficiently large $i\in A$, we say that $\aa$ has \emph{stable degree}, and we denote this stable degree by $\dg(\aa)$. 

We say that $f\in\kx$ is \emph{$\aa$-stable} if for some index $i\in A$, it satisfies
$$\ri(f)=\rj(f), \quad \mbox{ for all }\ j>i.$$
We obtain a \emph{stability function} $\rha$, defined
on the set of all $\aa$-stable polynomials by $\rha(f)=\max\{\ri(f)\mid i\in A\}$.

A \emph{limit key polynomial} for $\aa$ is a monic $\aa$-unstable polynomial of minimal degree. Let $\kpi(\aa)$ be the set of all these limit key polynomials. Since the product of stable polynomials is stable, all limit key polynomials are irreducible in $ \kx$.

The \emph{limit degree} of $\aa$, denoted  $\dgi(\aa)$,  is the degree of any limit key polynomial. If $\kpi(\aa)=\emptyset$, we agree that $\dgi(\aa)=\infty$. \e

\defn
We say that $\aa$ is an \emph{essential continuous family} of valuations in $\ttt$ if it has stable degree and  $\dg(\aa)<\dgi(\aa)<\infty$.\e

Take any limit key polynomial $\phi\in\kpi\left(\aa\right)$, and any $\ga\in\La\infty$ such that $\ri(\phi)<\ga$ for all $i\in A$.
We define the \emph{limit augmentation}   $\nu=[\aa;\,\phi,\ga]\in\ttt$ as the following mapping, defined in terms of $\phi$-expansions: 
$$ f=\sum\nolimits_{0\le s}a_s\phi^s\quad \deg(a_s)<\deg(\phi)\ \imp\  \nu(f)=\min\nolimits_{0\le s}\{\rha(a_s)+s\ga\}.
$$
Since $\deg(a_s)<\dgi(\aa)$, all coefficients $a_s$ are $\aa$-stable. 
 
Note that $\nu(\phi)=\ga$ and $\ri<\nu$ for all $i\in A$.
If $\ga<\infty$, then $\phi$ is a key polynomial for $\nu$ of minimal degree \cite[Cor. 7.13]{KP}. 

Consider a finite chain of valuations in $\ttt$
$$
\mu_0\ \stackrel{\phi_1,\ga_1}\lra\  \mu_1\ \stackrel{\phi_2,\ga_2}\lra\ \cdots
\ \lra\ \mu_{r-1} 
\ \stackrel{\phi_{r},\ga_{r}}\lra\ \mu_{r}=\nu
$$
in which every node is an augmentation  of the previous node, of one of the two types:\e

\emph{Ordinary augmentation}: \ $\mu_{n+1}=[\mu_n;\, \phi_{n+1},\ga_{n+1}]$, for some $\phi_{n+1}\in\kp(\mu_n)$.\e

\emph{Limit augmentation}:  \ $\mu_{n+1}=[\aa_n;\, \phi_{n+1},\ga_{n+1}]$,  for some $\phi_{n+1}\in\kpi(\aa_n)$, where $\aa_n$ is an essential continuous family containing $\mu_n$ as its first valuation.\e

We always consider an implicit choice of a key  polynomial $\phi_0\in\kp(\mu_0)$ of minimal degree, and we denote $\ga_0=\mu_0(\phi_0)$.


Therefore, for all $n$ such that $\ga_n<\infty$,  the polynomial $\phi_n$ is a key polynomial for $\mu_n$ of minimal degree, and we have
$$
m_n:=\deg(\mu_n)=\deg(\phi_n),\qquad \sval(\mu_n)=\mu_n(\phi_n)=\ga_n.
$$

\defn
A chain of mixed augmentations as above is said to be  a \emph{\mlv (MLV) chain}  if $\deg(\mu_0)=1$ and every augmentation step satisfies:
\begin{itemize}
	\item If $\,\mu_n\to\mu_{n+1}\,$ is ordinary, then $\ \deg(\mu_n)<\deg(\mu_{n+1})$.
	\item If $\,\mu_n\to\mu_{n+1}\,$ is limit, then $\ \deg(\mu_n)=\deg(\aa_n)$ and $\ \phi_n\not \in\ty(\mu_n,\mu_{n+1})$. 
\end{itemize}\e

In an MLV chain, all nodes $\mu_n$, with $n<r$, are residue transcendental valuations and satisfy $\phi_n\not \in\ty(\mu_n,\mu_{n+1})$. In particular, $\nu(\phi_n)=\ga_n$ for all $n$, by Lemma \ref{propertiesTMN}.

\begin{theorem}\cite[Thm. 4.3]{MLV}\label{main}
Let $\nu\in\ttt$ be an inner node or a finite leaf. Then, $\nu$ is the end node of a finite MLV chain.
\end{theorem}

These valuations are ``bien-specifi\'ees" in Vaqui\'e's terminology.

The main advantage of imposing the technical condition of MLV chain is that the nodes of the chain are essentially unique \cite[Thm. 4.7]{MLV}. Thus, we may read in the chain several data intrinsically associated to the valuation $\nu$.\e

\defn
The \emph{depth} of $\nu$ is the length of any MLV chain with end node  $\nu$. The \emph{limit-depth} of $\nu$ is the number of limit augmentations in this MLV chain.

An inner node or a finite leaf $\nu\in\ttt$ is said to be an \emph{inductive} valuation if it has limit-depth equal to zero.\e

The depth-zero valuations take the form $\nu=\om_{a,\dta}$, for some $a\in K$ and $\dta\in \La\infty$. They act as follows on $(x-a)$-expansions: 
$$
\nu\left(\sum\nolimits_{0\le s}a_s(x-a)^s\right) = \min\left\{v(a_s)+s\dta\mid0\le s\right\}.
$$

Clearly, $\om_{a,\infty}$ is a finite leave of $\ttt$ with support $(x-a)\kx$, while for $\dta<\infty$ the valuation $\om_{a,\dta}$ is an inner node admitting $x-a$ as a key polynomial. 
The valuation $\om_{a,\dta}$ is commensurable if and only if $\dta\in\gq\infty$. We clearly have
\begin{equation}\label{balls}
\om_{a,\dta}\le\om_{b,\ep}\ \sii\ \dta\le\min\{v(b-a),\ep\}.
\end{equation}

\subsection{Parametrization of equivalence classes of extensions of $v$ to $\kx$}
\label{subsecPrimitive}


Let $\g\hk \La$ be an extension  of ordered abelian groups, and let $\Delta\subset \La$ be the relative divisible hull of $\g$ in $\La$. We say that the extension  $\g\hk \La$ is {\it small} if $\La/\Delta$ is a cyclic group. 
For instance, for every valuation $\mu$ extending $v$ to $\kx$, the extension $\g\hk\gm$ is small \cite[Theorem 1.5]{Kuhl}. 

In \cite{csme}, a totally ordered real vector space $\rsme$ was constructed, which contains all small extensions of $\g$ up to $\g$-equivalence. This object is canonical, depending only on the set of nonzero principal convex subgroups of $\g$. However, the order-preserving embeddings $\g\hk\rsme$ are non-canonical, because they are obtained as an application of Hahn's embedding theorem \cite[Section 4]{csme}.

From now on, we fix any such embedding $\g\hk\rsme$ and we identify $\g$ with its image in $\rsme$. By the universal property of $\rsme$, for any small extension $\g\hk \La$, there exists an embedding $\La\hk\rsme$ fitting into a commutative diagram
$$
\begin{array}{ccc}
\La&&\\
\uparrow&\searrow&\\
\g&\hra&\rsme.
\end{array}
$$
In particular, every extension of $v$ to $\kx$ is equivalent to some extension of $v$ taking values in $\rsme$.

Consider two elements $x,y\in\rsme$ to be  equivalent if there exists an isomorphism of ordered groups between the subgroups $\gen{\g,x}$ and $\gen{\g,y}$, which maps $x$ to $y$ and acts as the identity on $\g$. 
There is a canonical set of representatives $\gsme\subset \rsme$ of this equivalence relation. We have 
$\g\subset\gq\subset\gsme\subset\rsme$. 

For each $x\in\gsme$ let us denote
\[
\left(\gq\right)_{\le x}=\left\{\ga\in\gq\mid \ga\le x\right\},\qquad
\left(\gq\right)_{\ge x}=\left\{\ga\in\gq\mid \ga\ge x\right\}
\]

The set $\gsme$ satisfies the following property, easily deduced from \cite[Lemma 5.4]{csme}.

\begin{proposition}\label{qcuts}
	The mapping $x\mapsto \left(\left(\gq\right)_{\le x},\left(\gq\right)_{\ge x}\right)$ establishes an order-preser\-ving isomorphism between $\gsme$ and the set of all quasicuts in $\gq$.  
\end{proposition}

A {\it quasicut} in $\gq$ is a pair of subsets $D=(D^L,D^R)$ such that $D^L\le D^R$ (every element in $D^L$ is less than or equal to every element  in $D^R$), $D^L\cup D^R=\gq$ and $D^L\cap D^R$ contains at most one element. The set of quasicuts admits a total ordering:
\[
\left(D^L,D^R\right)\le \left(E^L,E^R\right) \ \sii\ D^L\subset E^L\quad\mbox{and}\quad D^R\supset E^R.
\]



The following result follows immediately from Proposition \ref{qcuts}.

\begin{corollary}\label{sup}
Every subset $S\subset\gq$ admits a supremum in $\gsme$, which we simply denote by $\,\sup(S)$. If $S$ does not contain a maximal element, then $\sup(S)\not\in\gq$. 
\end{corollary}

Indeed, if $S$ contains a maximal element $\ga$, then $\sup(S)=\ga\in\gq$. Otherwise, $\sup(S)$ is the cut $(I,\gq\setminus I)$, where $I$ is the initial segment of $\gq$ generated by $S$.
Since $I$ contains no maximal element, necessarily $I\ne\left(\gq\right)_{\le \ga}$ for all $\ga\in\gq$.


Denote $\tq=\ttt(\gq)$. Consider the intermediate tree $\tq\subset\ts\subset\ttt(\rsme)$, defined as follows
$$\ts=\tq\cup\left\{\rho\in\ttt(\rsme)\mid \rho\mbox{ inner node with }\sval(\rho)\in\gsme\right\}.$$
Note that $\tq$ and $\ts$ have the same finite leaves.

The nodes of $\ts$ parametrize the equivalence classes of valuations on $\kx$ whose restriction to $K$ is equivalent to $v$ \cite[Thm. 7.1]{VT}. We need to recall some particular properties of the tree $\ts$ which will be useful in the sequel.


\subsubsection{Inner depth-zero nodes}

The inner depth-zero nodes in $\ts$ are of the form $\om_{a,\ga}$ for $a\in K$ and $\ga\in\gsme$.

Let  $-\infty=\min(\gsme)$ be the absolute minimal element in $\gsme$, corresponding to the quasicut $(\emptyset,\gq)$.
By (\ref{balls}), we have
\begin{equation}\label{dos}
\om_{a,-\infty}=\om_{b,-\infty}\le\om_{c,\ga} \quad\mbox{ for all }\ a,b,c\in K,\ \ga\in \gsme.
\end{equation}

\defn
Let us denote by $\minf=\om_{a,-\infty}$ this minimal depth-zero valuation, which is independent of $a\in K$.
By Theorem \ref{main} and (\ref{dos}), $\minf$ is an absolute minimal node of $\ts$. We say that $\minf$ is the \emph{root node} of $\ts$. \e

Since $\tq$ has no minimal node, the root node $\minf$ must be incommensurable. Hence, it has a unique tangent direction; actually, $$\kp(\minf)=\{x-a\mid a\in K\}=[x]_{\minf}.$$

All inner depth-zero nodes in $\ts$ are obtained as a single ordinary augmentation of the root node:
$$
\om_{a,\ga}=[\minf;\,x-a,\ga]\quad\mbox{ for all }a\in K,\ \ga\in\gsme,\ \ga>-\infty.
$$

\subsubsection{Limit augmentations}

Let $\aa=(\rho_i)_{i\in A}$ be an essential continuous family  in $\tq$, and let $\phi\in\kpi(\aa)$ be a limit key polynomial. 
By Corollary \ref{sup}, there exists a minimal limit augmentation of $\aa$ in $\ts$ with respect to $\phi$; namely
$$
\mu_\aa:=[\aa;\,\phi,\ga_\aa],\qquad \ga_\aa:=\sup\left\{\rho_i(\phi)\mid i\in A\right\}\in\gsme.
$$
Since $\aa$ has no maximal element, the family $\{\rho_i(\phi)\mid i \in A\}$ has no maximal element either. By Corollary \ref{sup}, $\ga_\aa$ does not belong to $\gq$. 


\begin{lemma}\cite[Lem. 7.2]{VT}\label{muaa}
The value $\ga_\aa\in\gsme\setminus \gq$ and the valuation $\mu_\aa\in\ts$ are independent of the choice of the limit key polynomial $\phi$ in $\kpi(\aa)$.
\end{lemma}

Since $\mu_\aa$ is incommensurable, it has a unique tangent direction. Actually,
$$\kp(\mu_\aa)=[\phi]_{\mu_\aa}= \left\{\phi+a\mid a\in\kx,\ \deg(a)<\deg(\phi), \ \rho_\aa(a)>\ga_\aa \right\}=\kpi(\aa).$$

Also, all limit augmentations $[\aa;\,\phi,\ga]$ are ordinary augmentations of $\mu_\aa$:
$$[\aa;\,\phi,\ga]=[\mu_\aa;\,\phi,\ga]\quad \mbox{ for all }\ga\in\gsme,\ \ga>\ga_\aa.$$

\subsubsection{Primitive nodes of $\ts$}

For the ease of the reader we shall consider the depth-zero valuations as a special case of limit augmentations.\e

\nn{\bf Convention. }We admit the empty set $\aa=\emptyset$ as an essential continuous family in $\tq$. We agree that this family has
$$
\ga_\aa=-\infty,\quad \mu_\aa=\minf,\quad \kpi(\aa)=\kp(\mu_\aa)=\left\{x-a\mid a\in K\right\}.
$$

Consider the subset $\kpmz\subset\kpm$ of \emph{strong} key polynomials, defined as
$$ 
\kpmz=\left\{\phi\in\kpm\mid \deg(\phi)>\deg(\mu)\right\}. 
$$
If $\kpmz\ne\emptyset$, then $\mu$ admits more than one tangent direction; thus, it is necessarily a residue transcendental valuation. In particular, $\sval(\mu)\in\gq$ and $\mu\in\tq$.\e 


\defn A \emph{limit-primitive} node in $\ts$ is the inner limit node $\mu_\aa$ associated to an essential continuous family $\aa$ in $\tq$. 

An \emph{ordinary-primitive\,} node in $\ts$ is an inner node $\mu\in\tq$ such that $\kpmz\ne\emptyset$. 

A \emph{primitive} node in $\ts$ is a node which is either limit-primitive or ordinary-primitive. Let us denote by  $\prim\left(\ts\right)$ the set of all primitive nodes. \e



For any inner node $\mu\in\ts$ and any $\phi\in\kpm$, consider the set of all ordinary augmentations of $\mu$ with respect to $\phi$:
$$
\pmph=\left\{[\mu;\,\phi,\ga]\mid \ga\in\gsme\infty,\ \ \mu(\phi)<\ga\le\infty\right\}\subset\ts.
$$

\defn Let $\rho\in \ts$ be a primitive node. Then, we define
$$
\prh=\begin{cases}
\bigcup_{\phi\in\kp_{\op{str}}(\rho)}\pset_\rho(\phi),&\mbox{ if $\rho$ is ordinary-primitive},\\\{\rho\}\cup
\bigcup_{\phi\in\kp(\rho)}\pset_\rho(\phi),&\mbox{ if $\rho$ is limit-primitive}.
\end{cases}
$$
 
\begin{theorem}\cite[Thm. 7.3]{VT}\label{previous}
Let $\nu\in\ts$ be either an inner node, or a finite leaf. There exists a unique primitive node $\rho\in\prim(\ts)$ such that $\nu\in\prh$.
\end{theorem}

\section{Okutsu equivalence of finite leaves}\label{secOk}

Let $(K,v)$ be a valued field. Consider an algebraic closure  of $K$ and the corresponding separable closure: $K\subset\ks\subset \kb$.
Let $\vb$ be an extension of $v$ to $\kb$, and let $(\kh,\vb)$ be the henselization of $(K,v)$ determined by $\vb$. Thus, $\kh\subset \ks$ is the fixed field of the decomposition group of $\vb$ in the Galois group $\gal(\ks/K)$.

In this section, we introduce a certain \emph{Okutsu equivalence} $\ok$ on the set $\lfin(\tq)$ of finite leaves of $\tq$ and we obtain a natural parametrization of the quotient set $\lfin(\tq)/\!\ok$ by a certain subset of the tangent space of $\ts$. 
Also, we use the tree structure of $\tq$ to define an ultrametric topology on the set $\lfin(\tq)$.

For any $\phi\in\irr(K)$, we denote by $K_\phi$ the simple field extension $K[x]/\phi\kx$.

\subsection{Finite leaves of $\tq$}\label{subsecFinLeaves}


Any $F\in\irr(\kh)$ determines a valuation $v_F$ on $\khx$ defined as:
$$
v_F(f)=v\left(f(\t)\right)\ \mbox{ for all }\ f\in\khx,
$$
where $\t\in\kb$ is a root of $F$. By the henselian property, the value $v\left(f(\t)\right)$ does not depend on the choice of $\t$. The support of $v_F$ is the prime ideal $F\khx$.

Let us denote by $w_F$ the restriction of $v_F$ to $\kx$. The support of $w_F$ is the prime ideal $\phi\kx$, where the ``norm" polynomial $\phi=N(F)\in\irr(K)$ is uniquely  determined by the equality $\left(F \khx\right)\cap \kx=\phi\kx$. Let $\wb_F$ be the valuation on $K_\phi$ naturally induced by $w_F$.

As a consequence of the results in \cite[Sec. 17]{endler}, the following mapping  is bijective: 
\begin{equation}\label{Bij}
\irr(\kh)\lra \lfin(\tq),\qquad F\longmapsto w_F. 
\end{equation}


Let us now recall how to describe the extensions of $v$ to the field $K_\phi$, for an arbitrary $\phi\in\irr(K)$. 
Since $\kh/K$ is a separable extension, the factorization of $\phi$ into a product of monic irreducible polynomials in $\kh[x]$ takes the form
$$
\phi=F_1\cdots F_r,
$$
with pairwise different $F_1,\dots,F_r\in\irr(\kh)$, whose norm is $N(F_i)=\phi$ \,for all $i$.

\begin{theorem}\label{Endler}\cite[Cor. 3.2]{NN}
The extensions  of $v$ to $K_\phi$ are precisely $\wb_{F_1},\dots,\wb_{F_r}$. 
\end{theorem}

\subsection{Primitive tangent space of $\ts$ and finite leaves}\label{subsecPrimTan}

By Theorem \ref{previous}, each finite leaf   $\f\in\lfin(\tq)$  belongs to $\prh$ for a unique primitive node $\rho\in\ts$. 
We say that $\rho$ is  the \emph{previous primitive node} of $\f$, and we denote it by $\rho=\rf\in\ts$. 

Since $\f$ is commensurable and has nontrivial support, we have $\rf\ne\f$. Indeed, the ordinary-primitive nodes have trivial support, while the limit-primitive nodes are incommensurable.
Consider a finite MLV chain whose end node is $\f$:
$$ \mu_0\ \stackrel{\phi_1,\ga_1}\lra\  \mu_1\ \lra\ \cdots\ \lra\ \mu_{r-1}\ \stackrel{\phi_{r},\infty}\lra\ \mu_{r}=\f.$$

If $\f$ has depth zero, then $\f=\mu_0=\om_{a,\infty}$ for some $a\in K$. Then, 
$$
\rf=\minf,\qquad \deg(\f)=\deg(\rf)=1.
$$

Suppose that $\f$ has a positive depth and the last augmentation $\mu_{r-1}\to\f$ is ordinary; that is, $\f=[\mu_{r-1};\,\phi_r,\infty]$.
By the definition of a MLV chain, $\phi_r\in\kp(\mu_{r-1})$ and  $\deg(\f)=\deg(\phi_r)>\deg(\mu_{r-1})$; thus, $\mu_{r-1}$ is a ordinary-primitive node and
\begin{equation}\label{dpthOrd}
\rf=\mu_{r-1},\qquad \deg(\f)=\deg(\phi_r)>\deg(\rf).
\end{equation}

Suppose that $\f$ has a positive depth and $\mu_{r-1}\to\f$ is a limit augmentation; that is, $\f=[\aa;\,\phi_r,\infty]$ for some  essential continuous family $\aa$ in $\tq$ whose first valuation is $\mu_{r-1}$. Then,  
$$
\rf=\mu_\aa,\qquad \deg(\f)=\deg(\phi_r)=\deg(\rf).
$$
\vskip0.1cm

\defn On the set $\lfin(\tq)$, we define the equivalence relation
$$
\f\ok\f'\ \sii\ \rf=\rfp\ \mbox{ and }\ \ty(\rf,\f)=\ty(\rf,\f').
$$
In this case, we say that the finite leaves $\f$ and $\f'$ are \emph{Okutsu equivalent}.

We denote by $\clok{\f}$ the Okutsu equivalence class of $\f$.\e

\defn Let  $\tprim$ be the subset of the tangent space of $\ts$ consisting of all tangent vectors based on primitive nodes:
$$
\tprim=\left\{(\rho,t)\mid \rho\in\prim(\ts),\ t\in\td(\rho)\right\}.
$$\vskip0.2cm

The next result follows immediately from Theorem \ref{previous} and the definition of $\ok$.

\begin{theorem}\label{mainT}
	There is a canonical bijective mapping
	$$
	\tprim\lra \lfin(\tq)/\!\ok,\qquad \left(\rho,[\phi]_\rho\right)\longmapsto \clok{[\rho;\,\phi,\infty]},
	$$
	whose inverse mapping is: \ $\clok{\f}\mapsto \left(\rf,\ty(\rf,\f)\right)$.
\end{theorem}

Combined with the identification $\lfin(\tq)=\irr(K)$ given by (\ref{Bij}), this result generalizes \cite[Thm. 5.14]{defless} to the most general case, where no assumption at all is made on the base valued field $(K,v)$.

Let us quote some basic properties of the equivalence relation $\ok$. 

\begin{lemma}\label{propertiesOk}
	Let $\f$, $\f'$ be two finite leaves in $\tq$, and let $\supp(\f)=\phi\kx$.
	\begin{enumerate}
		\item[(i)] The polynomial $\phi$ belongs to $\kp(\rf)$ and $\ty(\rf,\f)=[\phi]_{\rf}$.
		\item[(ii)] Let $\rho$ be a primitive node. Then, $\rho=\rf$ if and only if $\rho<\f$, $\phi\in\kpr$ and
\begin{equation}\label{deg}
\deg(\phi)>\deg(\rho),\quad\mbox{if $\rho$ is ordinary-primitive}. \end{equation}		
\item[(iii)] If $\rf$ is a limit-primitive node, then $\f\ok\f'$ if and only if $\rf=\rfp$.
		\item[(iv)] If  $\f\ok\f'$, then $\deg(\f)=\deg(\f')$.
	\end{enumerate}
\end{lemma}

\begin{proof}
By the definition of $\rf$, the valuation $\f=[\rf;\,\varphi,\ga]$ is an ordinary augmentation of $\rf$, for some $\varphi\in\kp(\rf)$, $\ga\in\gsme\infty$. Since $\f$ has nontrivial support, we have necessarily $\ga=\infty$. Then, $\varphi\in\kx$ is a monic irreducible polynomial such that $\supp(\f)=\varphi\kx$. This implies $\varphi=\phi$, and this proves (i). 

Let $\rho$ be a primitive node. If $\rho=\rf$, then $\rho<\f$ and $\phi\in\kpr$ by (i). Also, (\ref{deg}) follows from (\ref{dpthOrd}).

Conversely, suppose that $\rho<\f$ and $\phi\in\kpr$. Since $\rho(\phi)<\infty=\f(\phi)$, Lemma \ref{propertiesTMN} shows that $\ty(\rho,\f)=[\varphi]_\rho$ for some $\varphi\in\kpr$ such that $\varphi\mid_\rho\phi$. By \cite[Prop. 6.6]{KP}, we have $\varphi\sim_\rho\phi$, so that $\ty(\rho,\f)=[\phi]_\rho$.
This implies, $[\rho; \,\phi,\infty]=\f$ because both valuations coincide on $\phi$-expansions, again  by Lemma \ref{propertiesTMN}. 

If $\rho$ is limit-primitive, or $\rho$ is ordinary-primitive and (\ref{deg}) holds,  then $\f$ belongs to $\prh$ and this implies $\rho=\rf$ by Theorem \ref{previous}. This ends the proof of (ii).

If $\rf$ is a limit-primitive node, then it has a unique tangent direction; thus,  for all $\f'$ such that $\rho(\f)<\f'$ we have necessarily $\ty(\rf,\f)=\ty(\rf,\f')$. This proves (iii).
	
Finally, if  $\f\ok\f'$, then (iv) follows from (i), because
$\deg(\f)=\deg\,\ty(\rf,\f)=\deg\,\ty(\rf,\f')=\deg(\f')$.	
\end{proof}





\subsection{Ultrametric topology on the set of finite leaves}\label{subsecUltra}

Let us introduce an ultrametric topology on the set $\lfin(\tq)$.\e

\defn For all $\f,\gf\in\lfin(\tq)$, we define 
$$
u(\f,\gf)=\begin{cases}
\wt\left(\f\wedge \gf\right),&\mbox{ if }\f\ne \gf,\\
\infty,&\mbox{ if }\f=\gf,
\end{cases}
$$
where $\f\wedge \gf$ is the greatest common lower node in the tree $\tq$ \cite[Sec. 5.6]{VT}.  


\begin{lemma}\label{ultrametric}
	The function $u\colon \lfin(\tq)\times \lfin(\tq)\to \gq\infty$ has the following properties:
	\begin{enumerate}
		\item[(i)] $u(\f,\gf)=\infty  \ \sii \ \f=\gf$.
		\item[(ii)] $u(\f,\gf)\ge\min\{u(\f,\h),u(\h,\gf)\}\quad\mbox{for all}\quad \h\in \lfin(\tq)$.
		\item[(iii)] $u(\f,\gf)=u(\gf,\f)$.
	\end{enumerate}
\end{lemma}

\begin{proof}
	Conditions (i) and (iii) follow immediately from the definition of $u$.
	
	Let us prove (ii). If there is any coincidence between the three leaves $\f,\gf,\h$, then the statement of (ii) is obvious. Let us suppose that these leaves are pairwise different.
	
	Let $\rho=\h\wedge \f$, $\mu=\h\wedge \gf$. Since $u(\f,\gf)=u(\gf,\f)$, we may assume that $\rho\le\mu$. The relative position of all these nodes in the tree $\tq$ is the following:
	
	\begin{center}
		\setlength{\unitlength}{4mm}
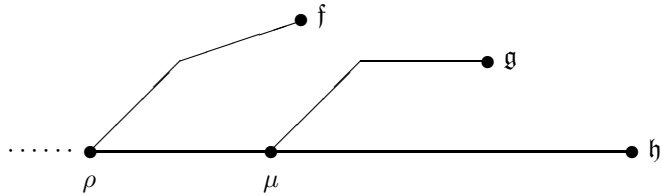
		\begin{picture}(20,7)
		\put(0.25,1.2){\line(1,0){18}}
		\put(0,0.9){$\bullet$}\put(6,0.9){$\bullet$}\put(18,0.9){$\bullet$}
		\put(7,5.3){$\bullet$}\put(13.2,3.9){$\bullet$}
		\put(-2.5,1){\begin{footnotesize}$\cdots\cdots$\end{footnotesize}}
		\put(18.8,1){\begin{footnotesize}$\h$\end{footnotesize}}
		\put(7.8,5.5){\begin{footnotesize}$\f$\end{footnotesize}}
		\put(14,4.1){\begin{footnotesize}$\gf$\end{footnotesize}}
		\put(6,0){\begin{footnotesize}$\mu$\end{footnotesize}}
		\put(0,0){\begin{footnotesize}$\rho$\end{footnotesize}}
		\put(0.25,1.2){\line(1,1){3}}\put(3.25,4.2){\line(3,1){4}}
		\put(6.25,1.2){\line(1,1){3}}\put(9.25,4.2){\line(1,0){4}}
		\end{picture}
	\end{center}\e
	
	We have $u(\f,\gf)=u(\f,\h)=\wt(\rho)$ and $u(\gf,\h)=\wt(\mu)$. On the other hand, Lemma \ref{propertiesTMN} shows that $\wt(\rho)<\wt(\mu)$, if $\rho<\mu$. Thus, condition (ii) holds.
\end{proof}\e

If $u(\f,\h)\ne u(\h,\gf)$, then $\rho<\mu$ and $u(\f,\gf)=\wt(\rho)=\min\{u(\f,\h),u(\h,\gf)\}$. The inequality $u(\f,\gf)>\min\{u(\f,\h),u(\h,\gf)\}$ holds in the following situation:

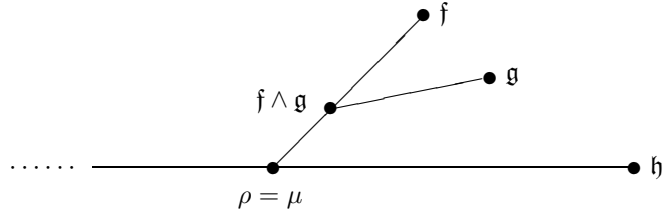
\begin{center}
	\setlength{\unitlength}{4mm}
	\begin{picture}(20,8)
	\put(0.25,1.2){\line(1,0){18}}
	\put(7.9,2.9){$\bullet$}\put(6,0.9){$\bullet$}\put(18,0.9){$\bullet$}
	\put(11,6){$\bullet$}\put(13.2,3.9){$\bullet$}
	\put(-2.5,1){\begin{footnotesize}$\cdots\cdots$\end{footnotesize}}
	\put(18.8,1){\begin{footnotesize}$\h$\end{footnotesize}}
	\put(11.8,6.1){\begin{footnotesize}$\f$\end{footnotesize}}
	\put(14,4.1){\begin{footnotesize}$\gf$\end{footnotesize}}
	\put(5.1,0){\begin{footnotesize}$\rho=\mu$\end{footnotesize}}
	\put(5.7,3.2){\begin{footnotesize}$\f\wedge \gf$\end{footnotesize}}
	\put(6.25,1.2){\line(1,1){5}}\put(8.2,3.15){\line(5,1){5}}
	\end{picture}
\end{center}\e

As a consequence, the function $u$ provides an structure of \emph{ultrametric space} on the set $\lfin(\tq)$, and a corresponding topology  \cite[Sec. 1.3]{KuhlBook}. A basis for the topology is formed by the balls
$$
B_\ga(\f):=\left\{\gf\in\lfin(\tq)\mid u(\f,\gf)>\ga\right\},\quad \f\in\lfin(\tq),\quad \ga\in\gq.
$$
In this topology, two finite leaves $\f,\gf$ are ``close" if the value $u(\f,\gf)$ is ``large".

Finally, let us remark that  two finite leaves in $\lfin(\tq)$  are Okutsu equivalent if and only if they are ``close enough" with respect to the ultrametric topology.  

\begin{lemma}\label{okTopology}
	Let $\f,\gf\in\lfin(\tq)$. Then, the following conditions are equivalent.
\begin{enumerate}
\item[(i)\ ] \ $\f\ok\gf$,
\item[(ii)\,] \ $\rf,\rg<\f\wedge\gf$,
\item[(iii)] \ $u(\f,\gf)>\max\{\wt(\rf),\wt(\rg)\}$.
\end{enumerate}
\end{lemma}

\begin{proof}
The equivalence between (i) and (ii) follows from Theorem  \ref{previous}. The equi\-va\-lence between (ii) and (iii) follows from Lemma \ref{propertiesTMN}.
\end{proof}



\section{Okutsu equivalence in the henselian case}\label{secOkHensel}

In this section, we suppose that $(K,v)$ is henselian. We fix an algebraic closure $\kb$ of $K$ and we still denote by $v$ the unique extension of $v$ to $\kb$.

Recall the canonical bijection between $\irr(K)$ and $\lfin(\tq)$ given in (\ref{Bij}):
\begin{equation}\label{BijHensel}
\irr(K)\lra \lfin(\tq),\qquad F\longmapsto v_F.
\end{equation}


Our first aim is to show that, under the identification $\lfin(\tq)=\irr(K)$, the ultrametric topology on the set $\lfin(\tq)$ introduced in Section \ref{secOk} coincides with the classical topology induced by $v$ on $\irr(K)$.

The following result was proved in \cite{defless} for defectless polynomials and extended to arbitrary irreducible polynimials in \cite[Thm. 4.4]{NN}.

\begin{theorem}\label{fundamental}
Let $F\in\irr(K)$ and $\phi\in\kpm$ for some valuation $\mu$ on $\kx$. Then, $$\phi\mmu F \ \sii\
\mu<v_F\ \mbox{ and }\ \,\ty(\mu,v_F)=\cl{\phi}.$$
In this case,
 $F\smu\phi^\ell$ with $\ell=\deg(F)/\deg(\phi)$. 
\end{theorem}

\begin{corollary}\label{muminimal}
Let $F\in\irr(K)$. Then, $F$ is $\mu$-minimal for all valuations $\mu<v_F$.
\end{corollary}

\begin{proof}
Let $\ty(\mu,v_F)=[\phi]_\mu$. Since $\mu(F)<\infty=v_F(F)$, Lemma \ref{propertiesTMN} shows that $\phi\mmu F$.
By Theorem \ref{fundamental}, $F\smu\phi^\ell$ with $\ell=\deg(F)/\deg(\phi)$. Since $\phi$ is $\mu$-minimal, Theorem \ref{univbound} shows that
$$
\dfrac{\mu(F)}{\deg(F)}=\dfrac{\mu(\phi)}{\deg(\phi)}=\wt(\mu).
$$
Hence, $F$ is $\mu$-minimal, again by Theorem \ref{univbound}.
\end{proof}\e




As another consequence of Theorem \ref{fundamental},  the ultrametric distance $u$ on $\irr(K)=\lfin(\ts)$ is given by a classical formula. 

\begin{corollary}\label{symmetry}
For all $F,G\in\irr(K)$ we have
$$
u(F,G):=u(v_F,v_G)=v_F(G)/\deg(G)=v\left(\op{Res}(F,G)\right),
$$
where $\op{Res}(F,G)$ is the resultant of the two polynomials.
\end{corollary}

\begin{proof}
If $F=G$, we have $v_F(F)=\infty=u(F,F)$. If $F\ne G$, then $\mu=v_F\wedge v_G$ is an inner node of  $\tq$ satisfying
$\mu<v_F$, $\mu<v_G$ and  $\ty(\mu,v_F)\ne\ty(\mu,v_G)$.

Let $\ty(\mu,v_F)=[\phi]_\mu$, $\ty(\mu,v_G)=[\varphi]_\mu$, so that $\phi\not\sim_\mu\varphi$. By Theorem \ref{fundamental},
$$
G\smu \varphi^\ell,\quad \ell=\deg(G)/\deg(\varphi).
$$
Hence, $\phi\nmid_\mu G$ and this implies $\mu(G)=v_F(G)$ by Lemma \ref{propertiesTMN}. On the other hand, $G$ is $\mu$-minimal by Corollary \ref{muminimal}. By Theorem \ref{univbound},
$$
\dfrac{v_F(G)}{\deg(G)}=\dfrac{\mu(G)}{\deg(G)}=\wt(\mu)=u(F,G).
$$

The equality $v_F(G)/\deg(G)=v\left(\op{Res}(F,G)\right)$ is well-known.
\end{proof}\e

In particular, the ultrametric topology on $\irr(K)$ coincides with the classical topology induced by the valuation $v$.\e

\defn
We denote by $\rho_F:=\rf$ the previous primitive node of the leaf $\,\f=v_F$.

\begin{corollary}\label{equivOk}
If $F,\,G\in \irr(K)$ have the same degree, then the following conditions are equivalent.
\begin{enumerate}
\item[(a)] $v_F\ok v_G$.
\item[(b)] $u(F,G)>\max\{\wt(\rho_F),\wt(\rho_G)\}$.
\item[(c)] $F\sim_{\rho_F}G$.
\end{enumerate}
\end{corollary}

\begin{proof}
Lemma \ref{okTopology} shows that conditions (a) and (b) are equivalent.

Since  $\deg(F)=\deg(G)$, we have $\deg(F-G)<\deg(F)$, so that
$$
\rho_F(F-G)=v_F(F-G)=v_F(G).
$$
Thus, condition (c) is equivalent to $v_F(G)>\rho_F(F)=\deg(F)\wt(\rho_F)$, which  is equi\-va\-lent to (b) by Corollary \ref{symmetry} and the symmetry of the argument.
\end{proof}\e

An obvious comparison of Corollary \ref{equivOk} with \cite[Lem. 5.13]{defless} shows that the restriction of the equivalence relation $\sim_{\op{ok}}$ to the subset $\op{Dless}(K)\subset\irr(K)$ of defectless polyomials, coincides with the Okutsu equivalence defined in \cite{defless}.\e


\defn
An $F\in\irr(K)$ is said to be \emph{defectless} if 
$\deg(F)=e(\vb_F/v)f(\vb_F/v)$,
where $\vb_F$ is the valuation on $K_F$ induced by $v_F$.
\e

Vaqui\'e characterized this property as follows  \cite{Vaq2}, \cite[Cor. 6.1]{MLV}. 

\begin{theorem}\label{VaqDef}
An $F\in\irr(K)$ is defectless if and only if $v_F$ is inductive.	
\end{theorem}

Since $v_F$ is an ordinary augmentation of its previous primitive node $\rho_F$, we see that $F$ is defectless if and only if $\rho_F$ is inductive.

Therefore, after identifying $\irr(K)=\lfin(\tq)$ through  (\ref{BijHensel}), Theorem \ref{mainT} yields a bijection between $\op{Dless}(K)/\!\ok$ and the following subset of $\tprim$:
$$
\T\sind:=\left\{(\rho,t)\in\tprim\mid \rho\ \mbox{ is inductive}\right\}.
$$
This subset $\T\sind$ may be easily identified with the Mac Lane space $\M$ of \cite{defless}. Therefore, even in the henselian case, Theorem \ref{mainT} extends \cite[Thm. 5.14]{defless} to a parametrization of Okutsu equivalence classes of arbitrary irreducible polynomials by a certain subset of the tangent space of $\ts$.  

\section{Okutsu frames over henselian fields}\label{secOkFrames}

We keep assuming that our valued field $(K,v)$ is henselian. Let us fix some $F\in\irr(K)$ of degree $n>1$. We define the \emph{weight} of a non-constant $g\in\kx$  as 
$$
\wt(g):=v_F(g)/\deg(g)\in\gq.
$$

\subsection{Definition of Okutsu frames}\label{subsecDefFrames}

For every integer $1<m\le n$, consider the set
$$
W_m(F)=\left\{\wt(g)\mid g\in\kx\mbox{ monic},\ 0<\deg(g)<m\right\}\subset\gq.
$$
The polynomial $F$ is defectless if and only if all these subsets $W_m(F)$ contain a maximal value \cite{Vaq2}, \cite[Thm. 5.7]{defless}.

In Section \ref{subsecPrimitive}, we defined a certain set $\gsme$ containing $\gq$, and we showed the existence of the following supremums: 
$$w_m(F):=\sup\left(W_m(F)\right)\in\gsme,\qquad 1<m\le n.$$

By Corollary \ref{sup}, if $W_m(F)$ does not contain a maximal element, then $w_m(F)\not\in\gq$.

\begin{lemma}\label{Krasner}
Let $\t\in \kb$ be a root of $F$ in $\kb$ and consider Krasner's constant
$$
\Omega(F)=\max\left\{v(\t-\t')\mid \t'\mbox{ root of }F, \ \t'\ne\t\right\}\in\gq.
$$

If $F$ is separable, then  \ $w_n(F)\le \Omega(F)$.
\end{lemma}

\begin{proof}
Suppose that $\wt(g)>\Omega(F)$ for some monic $g\in\kx$ such that $0<\deg(g)<n$.
Clearly, $\wt(g)=v(g(\t))/\deg(g)$ is the average of all $v(\t-\al)$ for $\al$ running in the multiset $Z(g)$ of roots of $g$ in $\kb$, counting multiplicities. Hence, we must have $v(\t-\al)>\Omega(F)$ for some $\al\in Z(g)$.

By Krasner's lemma, $\t$ is purely inseparable over $K(\al)$. Since $\t$ is separable over $K$, we must have $\t\in K(\al)$ and this contradicts the fact that $\deg(g)<n$.
\end{proof}\e

If $F$ is inseparable and has defect, then the set $W_n(F)$ may be unbounded in $\gq$ in which case, $w_n(F)=\max(\gsme)$ corresponds to the quasicut $(\gq,\emptyset)$. 

\begin{lemma}\label{phirr}
If there exists a maximal element in $W_n(F)$, then any monic $\phi\in\kx$ of minimal degree satisfying \,$\wt(\phi)=w_n(F)$
 is irreducible in $\kx$.
\end{lemma}

\begin{proof}
Suppose $\phi=ab$ for some monic $a,b\in\kx$ with $\deg(a),\deg(b)<\deg(\phi)$. By the minimality of $\deg(\phi)$, we have $\wt(a),\ \wt(b)<w_n(F)$; thus,$$w_n(F)=\wt(\phi)=\dfrac{v_F(a)+v_F(b)}{\deg(\phi)}<\dfrac{\deg(a)w_n(F)+\deg(b)w_n(F)}{\deg(\phi)}=w_n(F),$$which is a contradiction.
\end{proof}\e

Suppose that $\max(W_n(F))$ does not exist. For all weighted values $\be\in W_n$, let $\deg(\be)$ be  the minimal $\ell\in\N\cap[1,n)$ such that there exists a monic $g\in\kx$ of degree $\ell$ such that $\be=\wt(g)$.
Consider the minimal $m\in\N\cap[1,n)$ such that there exists a totally ordered cofinal family of constant degree $m$:
$$\bb=\left(\be_i\right)_{i\in B}\subset W_n(F),\quad \deg(\be_i)=m,\ \forall\,i\in B.$$
We may assume that the set of indices $B$ is well-ordered and the mapping $i\mapsto \be_i$ is an isomorphism of ordered sets between $B$ and our family $\left(\be_i\right)_{i\in B}$.

For all $i\in B$, choose a monic $\chi_i\in\kx$ of degree $m$ such that $\be_i=\wt(\chi_i)$.

Let $A\subset B$ be  the subset of all indices $i\in B$ such that $\chi_i$ is irreducible. Then the subfamily  $\left(\be_i\right)_{i\in A}$ is a final segment of $\bb$.

Indeed, by the minimality of $m$, there exists $\be_i\in\bb$ such that $\wt(g)<\be_i$ for all monic   
$g\in\kx$ of degree less than $m$.  Then, similar arguments to those used in the proof of Lemma \ref{phirr} show that $\chi_j$ is irreducible for all $j\ge i$. 

In particular, the family $\left(\be_i\right)_{i\in A}$ is still cofinal in  $W_n(F)$.

Consider the following set of monic irreducible polynomials of constant degree:
$$
\Phi=\begin{cases}
\left\{\chi_i\mid i\in A\right\},& \mbox{if }\, \nexists\max\left(W_n(F)\right),\\
\left\{\phi\right\},&  \wt(\phi)=\max\left(W_n(F)\right), \ \mbox{otherwise}.
\end{cases}
$$

Let $m$ be the constant degree of the polynomials in the set $\Phi$. We may apply this construction to find a set $\Phi'$  of monic irreducible polynomials of constant degree optimizing the weighted values in $W_m(F)$.
An iteration of this procedure leads to a finite sequence of such sets of polynomials
\begin{equation}\label{frame}
\left[\Phi_0,\Phi_1,\dots,\Phi_r,\Phi_{r+1}=\{F\}\right], 
\end{equation}
whose degrees grow strictly:
\begin{equation}\label{frameDeg}
1=m_0<m_1<\cdots<m_{r+1}=n,\qquad m_\ell=\deg\left(\Phi_\ell\right), \ \ 0\le\ell\le r+1,
\end{equation}
and the following property is satisfied: for any index $0\le\ell\le r$ and any monic polynomial $g\in\kx$ with $0<\deg(g)<m_{\ell+1}$, there exists $\phi\in\Phi_\ell$ such that
\begin{equation}\label{propertiesFrame}
\wt(g)\le  \wt(\phi).
\end{equation}

\defn An \emph{Okutsu frame} of $F$ is a list of monic irreducible polynomials as in (\ref{frame}), having degrees as in (\ref{frameDeg}), and satisfying the fundamental property (\ref{propertiesFrame}). \e

Clearly, we may replace each $\Phi_\ell$ with a suitable subset so that the following additional properties are satisfied, for all $0\le\ell\le r$: 

\begin{enumerate}
\item[(OF1)] \ $\#\Phi_\ell=1$ whenever $\max\left(W_{m_{\ell+1}}\right)$ exists.
\item[(OF2)] \ If $\max\left(W_{m_{\ell+1}}\right)$ does not exist, we may consider a total ordering on $\Phi_\ell$ determined by the action of $v_F$:
$$\phi<\phi'\ \sii \ v_F(\phi)<v_F(\phi').
$$
\item[(OF3)] \ For all $\phi\in\Phi_{\ell}$, $\varphi\in\Phi_{\ell+1}$, we have $\wt(\phi)<\wt(\varphi)$.
\end{enumerate}

From now on, we shall assume that our Okutsu frames satisfy these additional properties. Note that $w_{m_1}(F)<\cdots<w_{m_{r+1}}(F)=w_n(F)$ and

$$w_{m_{\ell+1}}(F)=
\begin{cases}
\wt(\phi),&\mbox{ if }\Phi_\ell=\{\phi\},\\
\sup\left\{\wt(\chi_i)\mid i\in A\right\},&\mbox{ if }\Phi_\ell=\{\chi_i\mid i\in A\}.
\end{cases}
$$

\subsection{Okutsu frames and Mac Lane--Vaqui\'e chains}\label{subsecOk=MLV}
In this section, we show that any MLV chain of $v_F$ determines an Okutsu frame of $F$. 

\begin{lemma}\label{maxmax}
Let  $\left[\Phi_0,\Phi_1,\dots,\Phi_r,\{F\}\right]$ be an Okutsu frame of $F$. For $0\le\ell\le r$, let
$$V_{m_\ell}=\left\{v_F(f)\mid f\in \kx\mbox{ monic of degree }m_\ell\right\}.$$
Then, there exists $\max\left(W_{m_{\ell+1}}\right)$ if and only if there exists $\max\left(V_{m_{\ell}}\right)$.
\end{lemma}

\begin{proof}
If $\max\left(W_{m_{\ell+1}}\right)$ exists, then $\Phi_{\ell}=\{\phi\}$ with $\max\left(W_{m_{\ell+1}}\right)=\wt(\phi)$, by the property (\ref{propertiesFrame}). Obviously, $\max\left(V_{m_{\ell}}\right)=v_F(\phi)$.

Suppose now that $\max\left(V_{m_{\ell}}\right)=v_F(\varphi)$ for some monic $\varphi\in\kx$ of degree $m_\ell$. By (\ref{propertiesFrame}), for all monic $f\in \kx$ with $\deg(f)<m_{\ell+1}$ we have
$\wt(f)\le \wt(\phi)$,
for some $\phi\in\Phi_\ell$. Since $v_F(\phi)\le v_F(\varphi)$, we deduce that  $\wt(\varphi)=
\max\left(W_{m_{\ell+1}}\right)$.
\end{proof}

\begin{theorem}\label{MLOk}
Let $F\in\irr(K)$ and consider a MLV chain of $v_F$:
$$
\mu_0\ \stackrel{\phi_1,\ga_1}\lra\  \mu_1\ \lra\ \cdots
\ \stackrel{\phi_{r-1},\ga_{r-1}}\lra\ \mu_{r-1}
\ \stackrel{\phi_{r},\ga_{r}}\lra\ \mu_{r}\ \stackrel{F,\infty}\lra\ \mu_{r+1}=v_F.
$$

Take $\Phi_{r+1}=\{F\}$. For each $0\le\ell\le r$, consider the following set of monic irreducible polynomials of constant degree:
$$
\Phi_\ell=\begin{cases}
\left\{\phi_\ell\right\}, &\mbox{ if }\ \mu_{\ell+1}=[\mu_\ell;\,\phi_{\ell+1},\ga_{\ell+1}]\mbox{ is an ordinary augmentation},\\
\left\{\chi_i\mid i\in A_\ell\right\}, &\mbox{ if }\ \mu_{\ell+1}=[\aa_\ell;\,\phi_{\ell+1},\ga_{\ell+1}]\mbox{ is a limit augmentation},
\end{cases}
$$
where $\aa_\ell=\left(\rho_i\right)_{i\in A_\ell}$ with $\rho_i=[\mu_\ell;\,\chi_i,v_F(\chi_i)]$ for all $i\in A_\ell$.
Then, $\left[\Phi_0,\dots,\Phi_r,\Phi_{r+1}\right]$ is an Okutsu frame of $F$. 
\end{theorem}

\begin{proof}
Let us first prove that the fundamental property (\ref{propertiesFrame}) holds for all $0\le \ell\le r$.

Suppose that $\mu_\ell\to\mu_{\ell+1}$ is an ordinary augmentation.
By the definition of a MLV chain, we have $m_\ell=\deg(\phi_\ell)<m_{\ell+1}=\deg(\phi_{\ell+1})$. Hence, for all monic $g\in\kx$ of degree $0<\deg(g)<m_{\ell+1}$, we have simultaneously $\phi_{\ell+1}\nmid_{\mu_\ell}\phi_\ell$ and $\phi_{\ell+1}\nmid_{\mu_\ell}g$. Since $\ty(\mu_\ell,v_F)=\left[\phi_{\ell+1}\right]_{\mu_\ell}$, Lemma \ref{propertiesTMN} shows that
$\mu_\ell(\phi_\ell)=v_F(\phi_\ell)$ and $\mu_\ell(g)=v_F(g)$.

Now, since $\phi_\ell$ is a key polynomial for $\mu_\ell$, Theorem \ref{univbound} implies
$$
\wt(g)=\dfrac{\mu_\ell(g)}{\deg(g)}\le \dfrac{\mu_\ell(\phi_\ell)}{m_\ell}=\wt(\phi_\ell).
$$
This ends the proof concerning all ordinary augmentation steps.

Suppose that $\mu_\ell\to\mu_{\ell+1}$ is a limit augmentation.  We mimic the above arguments, just by replacing the pair $\mu_\ell,\phi_\ell$ with the pair $\rho_i,\chi_i$ for a sufficiently large $i\in A_\ell$.

Let $\mal=[\aa_\ell;\,\phi_{\ell+1},\gaal]$ be the minimal limit augmentation of the essential continuous family $\aa_\ell$.
By the definition of a MLV chain, we have $$m_\ell=\deg(\chi_i)<m_{\ell+1}=\deg(\phi_{\ell+1}) \ \mbox{ for all }i\in A_\ell.$$ Hence, for all monic $g\in\kx$ of degree $0<\deg(g)<m_{\ell+1}$, we have simultaneously $\phi_{\ell+1}\nmid_{\mu_\ell}\chi_i$ and $\phi_{\ell+1}\nmid_{\mu_\ell}g$. Since $\ty(\mal,v_F)=\left[\phi_{\ell+1}\right]_{\mal}$, Lemma \ref{propertiesTMN} shows that
$$
\mal(g)=v_F(g),\qquad\mal(\chi_i)=v_F(\chi_i) \ \mbox{ for all }i\in A_\ell.
$$
On the other hand, $m_{\ell+1}=\dgi(\aa_\ell)$ is the unstable degree of $\aa_\ell$. Hence, $g$ and $\chi_i$ are $\aa_\ell$-stable. Take a sufficiently large $i$ such that
$$
\rho_i(g)=\ral(g)=\mal(g)=v_F(g).
$$
By \cite[Lem. 4.12]{VT}, we may assume that $\chi_j\nmid_{\rho_i} \chi_i$ for all $j>i$. This implies
$$
\rho_i(\chi_i)=\rho_j(\chi_i)=\ral(\chi_i)=\mal(\chi_i)=v_F(\chi_i),
$$
as well. Since $\chi_i$ is a key polynomial for $\rho_i$, Theorem \ref{univbound} implies
$$
\wt(g)=\dfrac{\rho_i(g)}{\deg(g)}\le \dfrac{\rho_i(\chi_i)}{m_\ell}=\wt(\chi_i).
$$
This ends the proof of (\ref{propertiesFrame}).

Finally, by \cite[Thm. 4.7]{MLV}, the augmentation step $\mu_\ell\to\mu_{\ell+1}$ is ordinary if and only if the set $V_{m_\ell}$ contains a maximal element. By Lemma \ref{maxmax}, $\#\Phi_\ell=1$ if and only if $W_{m_{\ell+1}}$  contains a maximal element. 
\end{proof}\e

In particular, the length $r+1$ of this Okutsu frame of $F$ is equal to the Mac Lane--Vaqui\'e depth of $v_F$.

Conversely, all Okutsu frames of $F$ arise in this way.

\begin{theorem}\label{OkML}
Let $\left[\Phi_0,\dots,\Phi_r,\Phi_{r+1}=\{F\}\right]$ be an Okutsu frame of $F\in\irr(K)$. For all $0\le \ell\le r+1$, choose an arbitrary $\,\phi_\ell\in\Phi_\ell$ and denote $\ga_\ell=v_F(\phi_\ell)$. Then, the truncation of $v_F$ by $\phi_\ell$ is a valuation $\mu_\ell$ fitting into a MLV chain of $v_F$:
$$\mu_0\ \stackrel{\phi_1,\ga_1}\lra\  \mu_1\ \lra\ \cdots
\ \stackrel{\phi_{r-1},\ga_{r-1}}\lra\ \mu_{r-1}
\ \stackrel{\phi_{r},\ga_{r}}\lra\ \mu_{r}\ \stackrel{F,\infty}\lra\ \mu_{r+1}=v_F.
$$

If $\Phi_\ell=\left\{\phi_\ell\right\}$, then $\mu_{\ell+1}=[\mu_\ell;\,\phi_{\ell+1},\ga_{\ell+1}]$ is an ordinary augmentation.

If $\Phi_\ell=\left\{\chi_i\mid i\in I_\ell\right\}$, then $\mu_{\ell+1}=[\aa_\ell;\,\phi_{\ell+1},\ga_{\ell+1}]$ is a limit augmentation with res\-pect to the essential continuous family $\aa_\ell=\{\mu_\ell\}\cup\left(\rho_i\right)_{i\in A_\ell}$, where $\rho_i=[\mu_\ell;\,\chi_i,v_F(\chi_i)]$ and $A_\ell\subset I_\ell$ contains the indices  $i$ such that $v_F(\chi_i)>\ga_\ell$.
\end{theorem}

\begin{proof}
Let $\phi_0=x-a$ for some $a\in K$. Since $\supp(v_F)=F\kx$ and $\deg(F)>1$, we have $\ga_0:=v_F(\phi_0)<\infty$. Thus,
the truncation of $v_F$ by $\phi_0$ is the depth-zero valuation $\mu_0=\om_{a,\ga_0}\le v_F$. Since $\mu_0$ has trivial support, we have $\mu_0<v_F$ and $\phi_0$ is a key polynomial for $\mu_0$ of minimal degree.
Also, the equality $\mu_0(\phi_0)=\ga_0=v_F(\phi_0)$ shows that $\phi_0\not\in\ty(\mu_0,v_F)$.

Now, suppose that for some $0\le \ell\le r$ we have constructed a MLV chain of the valuation $\mu_\ell$:
$$\mu_0\ \stackrel{\phi_1,\ga_1}\lra\  \mu_1\ \lra\ \cdots
\ \stackrel{\phi_{\ell},\ga_{\ell}}\lra\ \mu_{\ell}<v_F,
$$
satisfying all conditions of the theorem, for all indices $0,\dots,\ell$. Since $\mu_\ell\left(\phi_\ell\right)=\ga_\ell=v_F\left(\phi_\ell\right)$, we have $\phi_\ell\not\in\ty(\mu_\ell,v_F)$.

Let us construct a further step of the MLV chain, satisfying the conditions of the theorem for the index $\ell+1$ as well.

Denote $m_\ell=\deg(\phi_\ell)$ for all $0\le \ell\le r$.
 Suppose that $\Phi_\ell=\left\{\phi_\ell\right\}$. 
Since $\phi_\ell$ is a key polynomial for $\mu_\ell$ of  minimal degree,  Theorem \ref{univbound} shows that
$$
\dfrac{\mu_\ell(\phi_{\ell+1})}{m_{\ell+1}}\le \dfrac{\mu_\ell(\phi_{\ell})}{m_{\ell}}=\wt(\phi_{\ell})<\wt(\phi_{\ell+1}).
$$
Hence, $\mu_\ell(\phi_{\ell+1})<v_F(\phi_{\ell+1})$.
We claim that $m_{\ell+1}$ is the least degree of a monic polynomial in $\kx$ satisfying this inequality.
Indeed, suppose that $g\in\kx$ is a monic polynomial of minimal degree such that $\mu_\ell(g)<v_F(g)$. By Lemma \ref{propertiesTMN}, $g$ is a key polynomial for $\mu_\ell$; hence, it is $\mu_\ell$-minimal and  Theorem \ref{univbound} shows that
$$
\dfrac{\mu_\ell(g)}{\deg(g)}= \dfrac{\mu_\ell(\phi_{\ell})}{m_{\ell}}=\wt(\mu_\ell).
$$
If  $\deg(g)<m_{\ell+1}$,  the fundamental property (\ref{propertiesFrame}) would imply:
$$
\wt(g)\le\wt(\phi_{\ell})=\dfrac{\mu_\ell(\phi_{\ell})}{m_{\ell}}=\dfrac{\mu_\ell(g)}{\deg(g)},
$$
contradicting our initial assumption.

Thus, $\phi_{\ell+1}\in\kx$ is a monic polynomial of minimal degree satisfying $\mu_\ell(\phi_{\ell+1})<v_F(\phi_{\ell+1})$. By Lemma \ref{propertiesTMN}, $\phi_{\ell+1}$ is a key polynomial for $\mu_\ell$. The ordinary augmentation
$$
\mu_{\ell+1}:=[\mu_\ell; \phi_{\ell+1},\ga_{\ell+1}],\quad \ga_{\ell+1}=v_F\left(\phi_{\ell+1}\right),
$$
satisfies $\mu_\ell<\mu_{\ell+1}\le v_F$. If we add the augmentation step $\mu_\ell\to\mu_{\ell+1}$ to the previous MLV chain, we obtain a MLV chain of $\mu_{\ell+1}$, because
 $$\deg(\mu_\ell)=m_\ell<m_{\ell+1}=\deg(\mu_{\ell+1}).$$ 
If $\ell=r$, then $\phi_{\ell+1}=F$ and $\ga_{\ell+1}=\infty$, so that  $\mu_{\ell+1}=v_F$ and the theorem is proven.

If $\ell<r$, then $m_{\ell+1}<\deg(F)$ and this implies $\ga_{\ell+1}<\infty$. In this case, $\mu_{\ell+1}$ has trivial support and $\mu_{\ell+1}<v_F$. Also, $\phi_{\ell+1}$ is a key polynomial for $\mu_{\ell+1}$ of minimal degree. By Lemma \ref{minimal0}, the truncation of $v_F$ by $\phi_{\ell+1}$ is equal to:
$$
\left(v_F\right)_{\phi_{\ell+1}}= \left(\mu_{\ell+1}\right)_{\phi_{\ell+1}}=\mu_{\ell+1}.
$$
Finally, since $\mu_{\ell+1}\left(\phi_{\ell+1}\right)=\ga_{\ell+1}=v_F\left(\phi_{\ell+1}\right)$, we have necessarily $\phi_{\ell+1}\not\in\ty(\mu_{\ell+1},v_F)$. This ends the recurrence step in the case $\Phi_\ell=\left\{\phi_\ell\right\}$.

Now, suppose that $\Phi_\ell=\left\{\chi_i\mid i\in I\right\}$ and let $i_0\in I$ be the index for which $\phi_\ell=\chi_{i_0}$.

By our recurrence assumption, $\mu_\ell$ is a valuation and $\phi_\ell$ is a key polynomial for $\mu_\ell$ of minimal degree. Since $\deg\,\ty(\mu_\ell,v_F)\ge \deg(\mu_\ell)=m_\ell$, Lemma \ref{propertiesTMN} shows that
$$
\mu_\ell(a)=v_F(a)\quad \mbox{ for all }a\in\kx,\ \mbox{ with }\deg(a)<m_\ell,
$$
because $\phi\nmid_{\mu_\ell}a$ for any $\phi\in\ty(\mu_\ell,v_F)$.

Denote $\be_i=v_F(\chi_i)$ for all $i\in I$. By the additional property (i) of the Okutsu frame, we have $\be_{i_0}<\be_i$ for all $i_0<i$ in $I$. Since all $\chi_i\in\kx$ are monic polynomials of degree $m_\ell$, we deduce that
$$
\mu_\ell\left(\chi_{i_0}-\chi_i\right)=v_F\left(\chi_{i_0}-\chi_i\right)=\be_{i_0}=\sval(\mu_\ell)\quad\mbox{ for all }i_0<i.
$$
By \cite[Prop. 6.3]{KP}, all these $\chi_i\in\Phi_\ell$ are key polynomials for $\mu_\ell$ of degree $m_\ell$. Hence,  
we may consider the family of ordinary augmentations
$$
\rho_i=[\mu_\ell;\, \chi_i,\be_i]\quad\mbox{ for all }i>i_0.
$$
By comparing their action of $\chi_j$-expansions, we clearly have
$$
\mu_\ell<\rho_i<\rho_j<v_F\quad\mbox{ for all }\,i_0<i<j\,\mbox{ in }I.
$$

Denote $\rho_{i_0}:=\mu_\ell$.
For the totally ordered set of indices $A=I_{\ge i_0}$, we may consider a continuous family of valuations $\aa=\left(\rho_i\right)_{i\in A}$ of stable degree $m_\ell$.

Let us show that $\phi_{\ell+1}$ is $\aa$-unstable. For all $i\in A$,  Theorem \ref{univbound} shows that
$$
\dfrac{\rho_i\left(\phi_{\ell+1}\right)}{m_{\ell+1}}\le\dfrac{\sval(\rho_i)}{m_{\ell}}=\dfrac{\be_i}{m_{\ell}}<\wt(\phi_{\ell+1}),
$$
the last inequality by the additional property (OF3) of the Okutsu frame.
Hence, $\rho_i\left(\phi_{\ell+1}\right)<v_F(\phi_{\ell+1})$. By \cite[Cor. 2.5]{MLV}, this implies
$$
\rho_i\left(\phi_{\ell+1}\right)<\rho_j\left(\phi_{\ell+1}\right)\quad\mbox{ for all }\,i<j,
$$
so that $\phi_{\ell+1}$ is $\aa$-unstable. 

Let us now show that $m_{\ell+1}$ is the minimal degree of $\aa$-unstability; that is, all monic $g\in\kx$ with $\deg(g)<m_{\ell+1}$ are $\aa$-stable. Since the product of $\aa$-stable polynomials is $\aa$-stable, we may assume that $g$ is irreducible. By the fundamental property (\ref{propertiesFrame}), there exists $\chi_i\in\Phi_\ell$ such that
\begin{equation}\label{partial}
\wt(g)\le\wt(\chi_i)=\dfrac{\be_i}{m_{\ell}}=\wt(\rho_i).
\end{equation}
We claim that $\rho_i(g)=\rho_j(g)$ for all $j>i$. Indeed, suppose that $\rho_i(g)<\rho_j(g)\le v_F(g)$ for some $j>i$. On one hand, from (\ref{partial}) we deduce
$$
\dfrac{\ri(g)}{\deg(g)}<\wt(g)\le\wt(\rho_i),
$$
so that $g$ is not $\ri$-minimal, by Theorem \ref{univbound}. This contradicts Corollary \ref{muminimal}.

Therefore, $\phi_{\ell+1}$ is a monic $\aa$-unstable polynomial of minimal degree. In other words, $\phi_{\ell+1}\in\kpi(\aa)$.
The limit augmentation
$$
\mu_{\ell+1}:=[\aa; \phi_{\ell+1},\ga_{\ell+1}],\quad \ga_{\ell+1}=v_F\left(\phi_{\ell+1}\right),
$$
satisfies $\mu_{\ell+1}\le v_F$. Since $\phi_\ell\not\in\ty(\mu_\ell,\mu_{\ell+1})=\ty(\mu_\ell,v_F)$, if we add the augmentation step $\mu_\ell\to\mu_{\ell+1}$ to the previous MLV chain, we obtain a MLV chain of $\mu_{\ell+1}$.

If $\ell=r$, we have $\phi_{\ell+1}=F$ and $\ga_{\ell+1}=\infty$, so that  $\mu_{\ell+1}=v_F$ and the theorem would be proven.

If $\ell<r$, then $m_{\ell+1}<\deg(F)$ and this implies $\ga_{\ell+1}<\infty$. In this case, $\mu_{\ell+1}$ has trivial support and $\mu_{\ell+1}<v_F$. Also, $\phi_{\ell+1}$ is a key polynomial for $\mu_{\ell+1}$ of minimal degree. By Lemma \ref{minimal0}, the truncation of $v_F$ by $\phi_{\ell+1}$ is equal to:
$$
\left(v_F\right)_{\phi_{\ell+1}}= \left(\mu_{\ell+1}\right)_{\phi_{\ell+1}}=\mu_{\ell+1}.
$$
Finally, since $\mu_{\ell+1}\left(\phi_{\ell+1}\right)=\ga_{\ell+1}=v_F\left(\phi_{\ell+1}\right)$, we have necessarily $\phi_{\ell+1}\not\in\ty(\mu_{\ell+1},v_F)$. This ends the recurrence step in the case $\Phi_\ell=\left\{\chi_i\mid i\in I\right\}$. 
\end{proof}

\begin{corollary}\label{divide}
Let $\left[\Phi_0,\dots,\Phi_r,\Phi_{r+1}=\{F\}\right]$ be an Okutsu frame of some  $F\in\irr(K)$. Then, $1=m_0\mid m_1\mid\cdots\mid m_r\mid \deg(F)$.
\end{corollary}

\begin{proof}
By Theorem \ref{OkML}, $1=m_0,\dots,m_{r+1}=\deg(F)$ are the degrees of a MLV chain of $v_F$. Since $(K,v)$ is henselian,  all jumps $m_{\ell+1}/m_\ell$ take integer values.
Indeed, this follows from the results of Vaqui\'e \cite{Vaq2}, described in \cite[Sec. 6]{MLV} as well.
\end{proof}

\subsection{Computation of ramification indices, residual degrees and defect}
Since $(K,v)$ is henselian, the valuation $\vb_F$ is the only extension of $v$ to the field extension  $K_F=\kx/(F)$.
By Ostrowski's lemma, we have
$$
\deg(F)=e(F)f(F)d(F),
$$
where $e(F)=e(\vb_F/v)$ is the ramification index, $f(F)=f(\vb_F/v)$ the residual degree and $d(F)=d(\vb_F/v)$ the defect of $\vb_F/v$. Also, Ostrowski showed that $d(F)$ is a power of the exponent characeristic, defined as  
$$
p=
\begin{cases}
\chr(k),&\mbox{ if }\chr(k)>0,\\
1,&\mbox{ if }\chr(k)=0.
\end{cases}
$$

In \cite[Sec.6]{MLV}, it is shown how to compute these invariants $e(F)$, $f(F)$, $d(F)$ in terms of a MLV chain of $v_F$. Hence, as a consequence of Theorem \ref{OkML}, we can compute them in terms of an Okutsu frame.

Let $\left[\Phi_0,\dots,\Phi_r,\Phi_{r+1}=\{F\}\right]$ be an Okutsu frame of $F\in\irr(K)$. Consider the MLV chain of $v_F$ described in Theorem \ref{OkML}:
$$\mu_0\ \stackrel{\phi_1,\ga_1}\lra\  \mu_1\ \lra\ \cdots
\ \stackrel{\phi_{r-1},\ga_{r-1}}\lra\ \mu_{r-1}
\ \stackrel{\phi_{r},\ga_{r}}\lra\ \mu_{r}\ \stackrel{F,\infty}\lra\ v_F.
$$
Recall that $\ga_\ell=v_F(\phi_\ell)$ for all $\ell\ge0$.

By the MLV condition, this chain induces a chain of abelian groups:
$$
\g_{\mu_{-1}}:=\g\,\subset\,\g_{\mu_0}\,\subset\,\g_{\mu_1}\,\subset\,\cdots\,\g_{\mu_r}=\g_{v_F}=\g_{\vb_F},
$$
with $\g_{\mu_\ell}=\gen{\g,\ga_0,\dots,\ga_\ell}$ for all $0\le\ell\le r$ \cite[Sec. 4.1]{MLV}. Hence,
$$
e(F)=\left(\g_{\vb_F}\colon\g\right)=\left(\g_{\mu_r}\colon\g\right)=e_0\cdots e_r,
$$
where $e_\ell=\left(\g_{\mu_\ell}\colon\g_{\mu_{\ell-1}}\right)$ for all $0\le\ell\le r$. 

Thus, each index $e_\ell$ can be computed as the least positive integer $e$ such that $$e\,v_F(\phi_\ell)\in\gen{\g,v_F(\phi_0),\dots,v_F(\phi_{\ell-1})}.
$$

On the other hand, Vaqui\'e proved in \cite{Vaq2} that $d(F)=d_1\cdots d_r$, where
$$
d_\ell=\begin{cases}
1,&\mbox{ if }\mu_{\ell}\to\mu_{\ell+1}\mbox{ is an ordinary augmentation},\\
m_{\ell+1}/m_\ell,&\mbox{ if }\mu_{\ell}\to\mu_{\ell+1}\mbox{ is a limit augmentation},
\end{cases}
$$
for all $0\le\ell\le r$. Equivalently, 
$$
d_\ell=1\sii\#\Phi_\ell=1\sii \exists\max(W_{m_{\ell+1}}).
$$  
Thus, $d(F)$ may be computed solely in terms of the Okutsu frame too.

Finally, $f(F)=\deg(F)/e(F)d(F)$ is determined by  $e(F)$ and $d(F)$.

\subsection{Okutsu frames and abstract key polynomials}\label{subsecHOSKP}

Abstract key polynomials were introduced by Herrera-Olalla-Spivakovsky as an alternative approach to the methods of Mac Lane and Vaqui\'e, aiming at a thorough comprehension of the extensions to $\kx$ of arbitrary (not necessarily henselian) valuations on a field $K$
 \cite{hos,hmos}.
 
These polynomials were further studied by several authors and the comparison with Mac Lane-Vaqui\'e key polynomials is by now fully understood \cite{mahboub,Dec,NS2018,AFFGNR,N2021}.

Although they were classically defined only for valuations with trivial support, the paper \cite{AFFGNR} develops their properties for arbitrary valuations on $\kx$ too.
Let us recall a concrete comparison result between the two sorts of ``key polynomials".

\begin{theorem}\label{comparison}\cite[Thm. 2.21]{AFFGNR}
Let $\nu$ be a valuation on $\kx$ with nontrivial support. A monic polynomial $Q\in\kx\setminus K$ is an abstract key polynomial for $\nu$ if and only if either  $\supp(\nu)=Q\kx$, or the truncation $\nu_Q$ is a valuation and $Q$ is a (MLV) key polynomial of minimal degree for $\nu_Q$.
\end{theorem}

A set $\Psi$ of abstract key polynomials for $\nu$ is said to be \emph{complete} if for all non-constant $g\in\kx$ there exists $Q\in\Psi$ such that
\begin{equation}\label{complete}
\deg(Q)\le\deg(g)\quad\mbox{ and }\quad \nu_Q(g)=\nu(g).
\end{equation}

As a consequence of Theorem \ref{OkML}, we derive another interpretation of abstract key polynomials.

\begin{theorem}\label{HOS}
Suppose that $(K,v)$ is henselian and let $\left[\Phi_0,\dots,\Phi_r,\Phi_{r+1}=\{F\}\right]$ be an Okutsu frame of some  $F\in\irr(K)$. Then, the set $\,\Phi_0\cup\cdots\cup\Phi_r\cup\{F\}$ is a complete set of abstract key polynomials for $v_F$.
\end{theorem}

\begin{proof}
By Theorems \ref{OkML}	and \ref{comparison}, all polynomials in the set $\Phi_0\cup\cdots\cup\Phi_r\cup\{F\}$ are abstract key polynomials for $v_F$.

Take a monic $g\in\kx\setminus K$. If $\deg(g)\ge \deg(F)$, then (\ref{complete}) is satisfied for $Q=F$.

If $\deg(g)<\deg(F)$, then there exists $0\le\ell\le r$ such that $m_\ell\le\deg(g)<m_{\ell+1}$.

With the notation in Theorem \ref{OkML}, consider the MLV chain of $v_F$:
$$\mu_0\ \stackrel{\phi_1,\ga_1}\lra\  \mu_1\ \lra\ \cdots
\ \stackrel{\phi_{r-1},\ga_{r-1}}\lra\ \mu_{r-1}
\ \stackrel{\phi_{r},\ga_{r}}\lra\ \mu_{r}\ \stackrel{F,\infty}\lra\ \mu_{r+1}=v_F.
$$

If $\Phi_\ell=\{\phi_\ell\}$, then $\mu_{\ell+1}=[\mu_\ell; \phi_{\ell+1},\ga_{\ell+1}]$ is an ordinary augmentation and $\ty(\mu_\ell,v_F)=\ty(\mu_\ell,\mu_{\ell+1})=[\phi_{\ell+1}]_{\mu_\ell}$. Since $\deg(g)<\deg(\phi_{\ell+1})$, necessarily $\phi_{\ell+1}\nmid_{\mu_\ell}g$ and  $\mu_\ell(g)=v_F(g)$ by Lemma \ref{propertiesTMN}. Thus, (\ref{complete}) is satisfied for $Q=\phi_\ell$.

If $\Phi_\ell=\left\{\chi_i\mid i\in I_\ell\right\}$, then $\mu_{\ell+1}=[\aa_\ell;\,\phi_{\ell+1},\ga_{\ell+1}]$ is a limit augmentation and $\phi_{\ell+1}$ is a limit key polynomial for $\aa_\ell$. Hence, $g$ is $\aa_\ell$-stable and (\ref{complete}) is satisfied for $Q=\chi_i$, for a sufficiently large $i\in I_\ell$.	
\end{proof}\e

The converse statement follows easily from the definitions.
Consider a complete set of abstract key polynomials for $\nu:=v_F$, $$\Psi=\Psi_{m_0}\cup\cdots\Psi_{m_r}\cup\Psi_{m_{r+1}}=\{F\},\quad m_0=1<m_1<\cdots<m_r<\deg(F),$$  where all polynomials in $\Psi_{m_\ell}$ have degree $m_\ell$. 

In order to show that the list $[\Psi_1,\dots,\Psi_r,\{F\}]$ is an Okutsu frame of $F$, we need only to check that the property (\ref{propertiesFrame}) is satisfied.

Take a monic $g\in\kx$ such that $0<\deg(g)<m_{\ell+1}$ for some  $0\le \ell\le r$. By the completeness of $\Psi$, there exists $P\in\Psi$ such that $\deg(P)\le\deg(g)$ and $\nu_P(g)=\nu(g)$.
Since $\deg(P)\le m_\ell$, there exists $Q\in\Psi_{m_\ell}$ such that $\nu_P\le \nu_Q\le \nu$. Thus, $\nu_Q(g)=\nu(g)$ as well. By Theorem \ref{comparison}, $Q$ is a MLV key polynomial of minimal degree for $\nu_Q$. Thus, Theorem \ref{univbound} shows that
$$
\wt(g)=\dfrac{\nu(g)}{\deg(g)}=\dfrac{\nu_Q(g)}{\deg(g)}\le\dfrac{\nu_Q(Q)}{m_\ell}= \dfrac{\nu(Q)}{m_\ell}=\wt(Q).$$


\begin{thebibliography}{}

\bibitem{AFFGNR}M. Alberich-Carrami$\tilde{\mbox{n}}$ana, Alberto F. Boix, J. Fern\'andez, J. Gu\`ardia, E. Nart, J. Ro\'e, \emph{Of limit key polynomials}, Illin. J. Math. {\bf 65}, No.1 (2021), 201--229.

\bibitem{VT}M. Alberich-Carrami$\tilde{\mbox{n}}$ana,  J. Gu\`ardia, E. Nart, J. Ro\'e, \emph{Valuative trees of valued fields}, J. Algebra {\bf 614} (2023), 71--114.




\bibitem{Dec}J. Decaup, W. Mahboub, M. Spivakovsky, \emph{Abstract key polynomials and comparison theorems with the key polynomials of MacLane-Vaqui\'e}, Illin. J. Math. {\bf 62}, Number 1-4 (2018), 253--270.

\bibitem{endler}O. Endler, \emph{Valuation Theory}, Universitex, Springer-Verlag Berlin Heidelberg, 1972.






\bibitem{ResidualIdeals}J. Fern\'{a}ndez, J. Gu\`{a}rdia, J. Montes, E. Nart, \emph{Residual ideals of MacLane valuations}, J. Algebra {\bf 427} (2015), 30--75.



\bibitem{okutsu}
J. Gu\`{a}rdia, J.  Montes, E.  Nart, \emph{Okutsu invariants and Newton polygons}, Acta Arith. {\bf 145} (2010), 83--108.









\bibitem{hos}F.J. Herrera Govantes, M.A. Olalla Acosta, M. Spivakovsky, \emph{Valuations in algebraic field extensions}, J. Algebra {\bf 312} (2007), no. 2, 1033--1074.

\bibitem{hmos}F.J. Herrera Govantes, W. Mahboub, M.A. Olalla Acosta, M. Spivakovsky, \emph{Key polynomials for simple extensions of valued fields}, preprint, arXiv:1406.0657v4 [math.AG], 2018.


\bibitem{KuhlBook}F.-V. Kuhlmann, \emph{Valuation theory}, book in preparation. Preliminary version of several chapters available at {\tt http://math.usask.ca/\raise.6ex\hbox{\mbox{\tiny$\sim$}}fvk/Fvkbook.htm}


\bibitem{Kuhl}F.-V. Kuhlmann, \emph{Value groups, residue fields, and bad places of rational function fields}, Trans. Amer. Math. Soc. {\bf 356} (2004), no. 11, 4559--4660.


\bibitem{csme} F.-V. Kuhlmann, E. Nart, \emph{Cuts and small extensions of abelian ordered groups}, J. Pure Appl. Algebra {\bf 226} (2022), 107103.

\bibitem{mahboub}W. Mahboub, \emph{Key polynomials}, J. Pure Appl. Algebra {\bf 217} (2013), no. 6, 989--1006.

\bibitem{mcla} S. Mac Lane, \emph{A construction for absolute values in polynomial rings}, Trans. Amer. Math. Soc. {\bf40} (1936), pp. 363--395.

\bibitem{mclb} S. Mac Lane, \emph{A construction for prime ideals as absolute values of an algebraic field}, Duke Math. J. {\bf2} (1936), pp. 492--510.


\bibitem{defless} N. Moraes de Oliveira, E. Nart, \emph{Defectless polynomials over henselian fields and inductive valuations}, J. Algebra, {\bf 541} (2020), 270--307.





\bibitem{KP} E. Nart, \emph{Key polynomials over valued fields}, Publ. Mat. {\bf 64} (2020), 195--232.

\bibitem{MLV} E. Nart, \emph{MacLane-Vaqui\'e chains of valuations on a polynomial ring}, Pacific J. Math. {\bf 311-1} (2021), 165--195.


\bibitem{NN} E. Nart, J. Novacoski,  \emph{The defect formula}, arXiv:2207.11119 [math.AC], to appear in Adv. Math..



\bibitem{NS2018} J. Novacoski, M. Spivakovsky, \emph{Key polynomials and pseudo-convergent sequences}, J. Algebra {\bf 495} (2018), 199--219.


\bibitem{N2021} J. Novacoski, \emph{On MacLane--Vaqui\'e key polynomials}, J. Pure Appl. Algebra {\bf 225} (2021), 106644.

\bibitem{Ok}
K. Okutsu, \emph{Construction of integral basis I, II}, Proc. Japan Acad. Ser. A {\bf 58} (1982), 47--49, 87--89.

\bibitem{ore1} \O{}. Ore, \emph{Zur Theorie der algebraischen K\"orper}, Acta Math. {\bf44} (1923), pp. 219--314.







\bibitem{Vaq}
M. Vaqui\'e, \emph{Extension d'une valuation}, Trans. Amer. Math. Soc.  {\bf 359} (2007), no. 7, 3439--3481.

\bibitem{Vaq2}M. Vaqui\'e, \emph{Famille essential de valuations et d\'efaut d'une extension}, J. Algebra {\bf 311} (2007), no. 2, 859--876.



\end{thebibliography}
\end{document}